\documentclass{article}
\usepackage{a4wide}
\usepackage{amsthm}
\usepackage{graphicx}
\usepackage{amssymb}
\usepackage{amsmath}
\usepackage{ascmac}
\usepackage{setspace}
\usepackage{float}
\usepackage[dvips,usenames]{color}
\usepackage{colortbl}
\usepackage{algorithm}
\usepackage{algorithmic}
\usepackage{setspace}
\newtheorem{Theorem}[equation]{Theorem}

\newtheorem{Lemma}[equation]{Lemma}
\newtheorem{Proposition}[equation]{Proposition}

\theoremstyle{definition}
\newtheorem{Definition}[equation]{Definition}

\theoremstyle{remark}
\newtheorem{Remark}[equation]{Remark}
\numberwithin{equation}{section}
\newtheorem{Claim}[equation]{Claim}

\DeclareMathOperator{\ev}{ev}
\DeclareMathOperator{\id}{id}
\DeclareMathOperator{\tr}{tr}

\DeclareMathOperator{\row}{row}
\DeclareMathOperator{\col}{col}
\newcommand{\ve}{\varepsilon}

\allowdisplaybreaks

\begin{document}
\title{Twisted Affine Yangian and Rectangular $W$-algebra of type $D$}
\author{Mamoru Ueda}
\date{}
\maketitle
\begin{abstract}
We define the twisted affine Yangian of type $C$ and construct surjective homomorphisms from twisted affine Yangians of type $C$ to the universal enveloping algebra of the rectangular $W$-algebra associated with $\mathfrak{so}(ln)$ and a nilpotent element whose Jordan form corresponds to the partition $(l^n)$ in the case when $l$ and $n$ are even.\footnote[0]{{\bf Key words;} quantum group, affine Yangian, vertex algebra, $W$-algebra\\{\bf 2020 Math. Subject
Classification;} 17B37\\
\qquad{\bf Institute;} Research Institute for Mathematical Sciences, Kyoto, JAPAN\\
\qquad{\bf email adress;} udmaoru@kurims.kyoto-u.ac.jp}
\end{abstract}
\section{Introduction}

In this article, we extend our previous work \cite{U4} to type $D$ setting. 

Drinfeld (\cite{D1}, \cite{D2}) introduced the finite Yangian in order to solve the Yang-Baxter equation. 
The finite Yangian $Y_h(\mathfrak{g})$ is a quantum group associated with a finite dimensional simple Lie algebra $\mathfrak{g}$ and one parameter $h\in\mathbb{C}$ and is a deformation of a current algebra $\mathfrak{g}\otimes\mathbb{C}[z]$. The relationships between finite Yangians of type $A$ and finite $W$-algebras (\cite{Pr}) of type $A$ has been studied (see \cite{RS}, \cite{BK}). A finite $W$-algebra $\mathcal{W}^{\text{fin}}(\mathfrak{g},f)$ is an associative algebra associated with a finite dimensional reductive Lie algebra $\mathfrak{g}$ and a nilpotent element $f\in\mathfrak{g}$. Ragoucy and Sorba \cite{RS} gave surjective homomorphisms from finite Yangians of type $A$ to finite rectangular $W$-algebras of type $A$. More generally, in \cite{BK}, Brundan and Kleshchev constructed a surjective homomorphism from a shifted Yangian, a subalgebra of the finite Yangian of type $A$, to an arbitrary finite $W$-algebra of type $A$.  
For type $CD$ cases, Brown \cite{Bro} constructed surjective homomorphisms from twisted Yangians to rectangular finite $W$-algebras of type $CD$ by using twisted Yangians instead of Yangians. Twisted Yangians were introduced by Olshanskii (\cite{O}) and were further studied in \cite{GR, M1, MNO} etc. The twisted Yangian $T_h(\mathfrak{g},\mathfrak{k})$ is an associative algebra associated with one parameter $h$, a finite dimensional simple Lie algebra $\mathfrak{g}$, subspaces $\mathfrak{k},\mathfrak{m}\subset\mathfrak{g}$, and a symmetric involution $\theta\colon\mathfrak{g}\to\mathfrak{g}$ such that $\mathfrak{g}^\theta=\mathfrak{k}$ and $\mathfrak{m}=\{x\in\mathfrak{g}\mid\theta(x)=-x\}$. The twisted Yangian $T_h(\mathfrak{g},\mathfrak{k})$ can be realized as a coideal of the finite Yangian $Y_h(\mathfrak{g})$. 

A finite $W$-algebra $\mathcal{W}^{\text{fin}}(\mathfrak{g},f)$ can be regarded as a finite analogue of a $W$-algebra $\mathcal{W}^k(\mathfrak{g},f)$ (\cite{DSK1}, \cite{A1}). A $W$-algebra $\mathcal{W}^k(\mathfrak{g},f)$ is a vertex algebra associted with a finite dimensional reductive Lie algebra $\mathfrak{g}$ and a nilpotent element $f\in\mathfrak{g}$. In the case when $\mathfrak{g}$ is ${\mathfrak{sl}}(n)$, there exists a similar result to that of Ragoucy-Sorba \cite{RS} in the affine setting. The corresponding Yangian is the two parameter's Yangian $Y_{\ve_1,\ve_2}(\widehat{\mathfrak{sl}}(n))$, which was defined by Guay (\cite{Gu1}, \cite{Gu2}). We call $Y_{\ve_1,\ve_2}(\widehat{\mathfrak{sl}}(n))$ Guay's affine Yangian.  It has been shown that Guay's affine Yangian is a deformation of the central extension of the current algebra ${\mathfrak{sl}}(n)[u^{\pm1},v]$ and has a Hopf algebra structure (\cite{Gu1}, \cite{GNW}). Precisely speaking, there exists an algebra homomorphism satisfying the coassociativity;
\begin{equation*}
\widetilde{\Delta}\colon \widetilde{Y}_{\ve_1,\ve_2}(\widehat{\mathfrak{sl}}(n))\to Y_{\ve_1,\ve_2}(\widehat{\mathfrak{sl}}(n))\widehat{\otimes}Y_{\ve_1,\ve_2}(\widehat{\mathfrak{sl}}(n)),
\end{equation*}
where $\widetilde{Y}_{\ve_1,\ve_2}(\widehat{\mathfrak{sl}}(n))$ and $Y_{\ve_1,\ve_2}(\widehat{\mathfrak{sl}}(n))\widehat{\otimes}Y_{\ve_1,\ve_2}(\widehat{\mathfrak{sl}}(n))$ are the standard degreewise completions of $Y_{\ve_1,\ve_2}(\widehat{\mathfrak{sl}}(n))$ and $Y_{\ve_1,\ve_2}(\widehat{\mathfrak{sl}}(n))\otimes Y_{\ve_1,\ve_2}(\widehat{\mathfrak{sl}}(n))$.
The affine Yangian associated with $\widehat{\mathfrak{gl}}(1)$ was defined by using a geometric realization of the Guay's affine Yangian (\cite{N1}, \cite{V}, \cite{SV}).
In the case that $f$ is the principal nilpotent element, Schiffmann and Vasserot (\cite{SV}) have constructed a surjective homomorphism from the Yangian of $\widehat{\mathfrak{gl}}(1)$ to the universal enveloping algebras (see \cite{FZ} and \cite{MNT}) of the principal $W$-algebras of type $A$ and have proved the celebrated AGT conjecture (\cite{Ga}, \cite{BFFR}). In our previous work \cite{U4}, we gave a homomorphism from Guay's affne Yangian $Y_{\ve_1,\ve_2}(\widehat{\mathfrak{sl}}(n))$ to the universal enveloping algebra of $\mathcal{W}^{k}(\mathfrak{gl}(nl),(l^n))$, the rectangular $W$-algebra associated with $\mathfrak{g}=\mathfrak{gl}(nl)$ and its nilpotent element $f$ whose Jordan form corresponds to the partition $(l^n)$. In \cite{U4}, we also show the corresponding statement in the super setting. That is, we have shown that there exists a surjective homomorphism from the affine super Yangian $Y_{\ve_1,\ve_2}(\widehat{\mathfrak{sl}}(m|n))$ defined in \cite{U2} to the universal enveloping algebra of rectangualar $W$-superalgebra $\mathcal{W}^{k}(\mathfrak{gl}(ml|nl),(l^{(m|n)}))$. Moreover, in \cite{KU}, an another proof to the main result of \cite{U4} was given. Precisely speaking, we construct these homomorphisms by using the coproduct and the evaluation map for the Guay's affine Yangian.

In this paper, we deal with the rectangular $W$-algebra $\mathcal{W}^k(\mathfrak{so}(nl),(l^{n}))$, the $W$-algebra associated with $\mathfrak{g}=\mathfrak{so}(nl)$ and a nilpotent element $f$ whose Jordan form corresponds to the partition $(l^n)$ in the case when $l$ and $n$ are even. The corresponding Yangian is the twisted affine Yangian $TY_{\ve_1,\ve_2}(\widehat{\mathfrak{sp}}(n))$ which is defined by using the Drinfeld $J$ presentation of the Guay's affine Yangian in the sense of \cite{GNW}. The Drinfeld $J$ presentation of the finite Yangian is Drinfeld's original definition of the finite Yangian (\cite{D1}) whose generators are $\{x,J(x)\mid x\in\mathfrak{g}\}$, where $J(x)$ is corresponding to $x\otimes z\in\mathfrak{g}\otimes\mathbb{C}[z]$. Reffering to the Drinfeld $J$ presentation of $Y_h(\mathfrak{g})$, Belliard and Regelskis (\cite{BR1}) constructed the Drinfeld $J$ presentation of the twisted Yangian whose generators are $\{x,B(y)\mid x\in\mathfrak{k},y\in\mathfrak{m}\}$, where $B(y)$ is corresponding to $y\otimes z\in\mathfrak{m}\otimes\mathbb{C}[z]$ when we set $h=0$. In \cite{GNW}, Guay-Nakajima-Wendland constructed the terms $J(h_i),J(x^\pm_i)\in\widetilde{Y}_{\ve_1,\ve_2}(\widehat{\mathfrak{sl}}(n))$ in the analogy of the Drinfeld $J$ presentation of the finite Yangians. We define $TY_{\ve_1,\ve_2}(\widehat{\mathfrak{sp}}(n))$ as a subalgebra of $\widetilde{Y}_{\ve_1,\ve_2}(\widehat{\mathfrak{sl}}(n))$ in terms of $J(h_i)$. We note that $TY_{\ve_1,\ve_2}(\widehat{\mathfrak{sp}}(n))$ becomes a coideal of $\widetilde{Y}_{\ve_1,\ve_2}(\widehat{\mathfrak{sl}}(n))$.

We construct a surjective homomorphism from the twisted affine Yangian $TY_{\ve_1,\ve_2}(\widehat{\mathfrak{sp}}(n))$ to the universal enveloping algebra of $\mathcal{W}^k(\mathfrak{so}(nl),(l^{n}))$ in the case when $l$ and $n$ are even. 
\begin{Theorem}\label{t1}
Let $n\geq4$ and $l$ be positive even. For any $k\in\mathbb{C}$, we set
\begin{align*}
\ve_1=-\dfrac{(k+(l-1)n-2)\hbar}{n},\quad\ve_2=\hbar+\dfrac{(k+(l-1)n-2)\hbar}{n}.
\end{align*}
There exists an algebra homomorphism
\begin{equation*}
\Phi\colon TY_{\ve_1,\ve_2}(\widehat{\mathfrak{sp}}(n))\to\mathcal{U}(\mathcal{W}^k(\mathfrak{so}(nl),(l^{n}))).
\end{equation*}
 Moreover, the homomorphism $\Phi$ is surjective
 provided that $k+(l-1)n-2\neq0$. 
\end{Theorem}

By Theorem~\ref{t1}, any (irreducible) representation of $\mathcal{W}^k(\mathfrak{so}(nl),(l^{n}))$ can be pulled back as that of $TY_{\ve_1,\ve_2}(\widehat{\mathfrak{sp}}(n))$. We note that the homomorphism $\Phi$ can be written by using the coproduct and the evaluation map for the Guay's affine Yangian as in \cite{KU}.
\section*{Acknowledgement}
The author wishes to express his gratitude to his supervisor Tomoyuki Arakawa for suggesting lots of advice to improve this paper. The author is particularly grateful for the assistance given by Naoki Genra. This work was supported by Iwadare Scholarship and and JSPS KAKENHI, Grant-in-Aid for JSPS Fellows, Grant Number JP20J12072. 
\section{Rectangular $W$-algebras of type $D$}
For all $n\in\mathbb{Z}_{>0}$,
let $I_n$ be $\{-n+1, -n+3, \dots, n-1\}$. Then, $\mathfrak{gl}(n)$ has a basis $\{e_{i,j}\mid i,j\in I_n\}$, where $e_{i,j}$ is a matrix unit. Using an $n \times n$ matrix $J_n\in\mathfrak{gl}(n)$ whose $(i,j)$ component is equal to $\delta_{i,-j}$, we can set $\mathfrak{so}(n)$ as $\{x \in \mathfrak{gl}(n)\mid x^T J_n + J_n x = 0\}$,
where $x^T$ is the transpose of $x$. We remark that $\mathfrak{so}(n)$ is not simple but reductive in the sense of this definition.
Under this notation, $\mathfrak{so}(n)$ is spanned by the set of matrices $\{f_{i,j}=e_{i,j} - e_{-j,-i}\mid i,j \in I_n\}$.

In this paper, we suppose that $l$ and $n$ are even positive. For all $a\in I_{nl}$, we take $\row(a)\in I_n$ and $\col(a)\in I_l$ such that $a=(\col(a))n+\row(a)$.
By the definition of $\row(a)$ and $\col(a)$, we have $\row(-a) = -\row(a)$ and $\col(-a) = -\col(a)$. 

We take a nilpotent element $f$ as follows;
\begin{equation*}
f=\sum_{\substack{a,b \in  I_{nl} \\ \row(a) = \row(b) \\ 
           \col(b) +2= \col(a) \geq 2} }
       f_{a,b}+\sum_{\substack{a,b \in  I_{nl} \\ \row(a) = \row(b)>0 \\ 
           \col(b) +2= \col(a)=1} }
       f_{a,b}.
\end{equation*}
We also set 
\begin{equation*}
\mathfrak{g}_p=\bigoplus_{\substack{a,b\in I_{nl},\\\col(b)-\col(a)=p}}\mathbb{C}f_{a,b}\subset\mathfrak{so}(nl).
\end{equation*}
and fix the $\mathfrak{sl}_2$-triple $(x, e, f)$ such that
\begin{equation*}
\mathfrak{g}_p=\{y\in\mathfrak{so}(nl)\mid[x,y]=py\}.
\end{equation*}
Let $\mathfrak{b}=\displaystyle\bigoplus_{r \leq 0}\limits \mathfrak{g}_r$ and $\mathfrak{c}=\displaystyle\bigoplus_{r <0}\limits\mathfrak{g}_r$, then $\mathfrak{b}$ and $\mathfrak{c}$ are subalgebras of $\mathfrak{so}(nl)$. We take an invariant inner product on $\mathfrak{so}(nl)$ by
\begin{align*}
\quad(f_{a_1,b_1},f_{a_2,b_2})
&=k(\delta_{a_1,b_2}\delta_{b_1,a_2}-\delta_{a_1+a_2,0}\delta_{b_1+b_2,0}).
\end{align*}

We fix some notations about vertex algebras. For a vertex algebra $V$, we denote the generating field associated with $v\in V$ by $v(z)=\displaystyle\sum_{n\in\mathbb{Z}}\limits v_{(n)}z^{-n-1}$ and the vacuum vector (resp. the translation operator) by $|0\rangle$ (resp. $\partial$). We also denote the OPE of $u,v\in V$ by
\begin{equation*}
u(z)v(w)\sim\displaystyle\sum_{s\geq0}\limits \dfrac{(u_{(s)}v)(w)}{(z-w)^{s+1}}.
\end{equation*}
There exists an inner product on $\mathfrak{so}(nl)$ determined by
\begin{align*}
\kappa(f_{a_1,b_1},f_{a_2,b_2})
&=(\delta_{a_1,b_2}\delta_{b_1,a_2}-\delta_{a_1+a_2,0}\delta_{b_1+b_2,0})\alpha+\delta_{a_1,b_1}\delta_{a_2,b_2}(\delta_{\col(a_1),\col(a_2)}-\delta_{\col(a_1)+\col(a_2),0}),
\end{align*}
where $\alpha=k+(l-1)n-2$. Let $\widehat{\mathfrak{b}}=\mathfrak{b}[t^{\pm1}]\oplus\mathbb{C}y$ be the affinization of $\mathfrak{b}$ associated with the inner product $\kappa$. We define a left $\widehat{\mathfrak{b}}$-module $V^\kappa(\mathfrak{b})$ as $U(\widehat{\mathfrak{b}})/U(\widehat{\mathfrak{b}})(\mathfrak{b}[t]\oplus\mathbb{C}(y-1))\cong U(\mathfrak{b}[t^{-1}]t^{-1})$. 
Then, $V^\kappa(\mathfrak{b})$ has a vertex algebra structure whose vacuum vector is $1$ and the generating field $(ut^{-1})(z)$ is equal to $\displaystyle\sum_{s\in\mathbb{Z}}\limits(ut^s)z^{-s-1}$ for all $u\in\mathfrak{b}$. We denote the generating field $(ut^{-1})(z)$ also by $u(z)$. We call $V^\kappa(\mathfrak{b})$ the universal affine vertex algebra associated with $(\mathfrak{b},\kappa)$.
By the definition of $V^\kappa(\mathfrak{b})$, generating fields $u(z)$ and $v(z)$ satisfy
\begin{gather}
u(z)v(w)\sim\dfrac{[u,v](w)}{z-w}+\dfrac{\kappa(u,v)}{(z-w)^2}\label{OPE1}
\end{gather}
for all $u,v\in\mathfrak{b}$.

Let $\mathfrak{a}$ be a Lie superalgebra generated by $\{J^{(u)},\psi_v\mid u\in\mathfrak{b},v\in \mathfrak{c}\}$ with the following commutator relations;
\begin{gather*}
[J^{(u)},J^{(v)}]=J^{([u,v])},\quad [J^{(u)},\psi_{v}]=\psi_{[u,v]},\quad [\psi_u,\psi_v]=0,
\end{gather*}
where $J^{(u)}$ is an even element and $\psi_v$ is an odd element. We define a vertex algebra $V^{\widetilde{\kappa}}(\mathfrak{a})$ associated with a Lie superalgebra $\mathfrak{a}$ and the inner product on $\mathfrak{a}$ determined by
\begin{gather*}
\widetilde{\kappa}(J^{(u)},J^{(v)})=\kappa(u,v),\quad \widetilde{\kappa}(J^{(u)},\psi_v)=\widetilde{\kappa}(\psi_u,\psi_v)=0.
\end{gather*}
In this section, we regard $V^{\widetilde{\kappa}}(\mathfrak{a})$ (resp. \ $V^\kappa(\mathfrak{b})$) as a non-associative superalgebra whose product $\cdot$ is defined by $u\cdot v=u_{(-1)}v$.
In order to simplify the notation, we denote $J^{(u)}t^{s}\in V^{\widetilde{\kappa}}(\mathfrak{a})\text{ or }V^\kappa(\mathfrak{b})$ by $u[s]$ and set
\begin{gather*}
\widehat{i}=\begin{cases}
0\quad\text{ if }i\geq0,\\
1\quad\text{ if }i<0.
\end{cases}
\end{gather*}
By \cite{KRW}, $\mathcal{W}^k(\mathfrak{so}(nl),(l^{n}))$ can be realized as a vertex subalgebra of $V^\kappa(\mathfrak{b})$.
\begin{Definition}\label{T125}
We define $\mathcal{W}^k(\mathfrak{so}(nl),(l^n))$ as
\begin{equation*}
\mathcal{W}^k(\mathfrak{so}(nl),(l^n))=\{y\in V^\kappa(\mathfrak{b})\mid d_0(y)=0\},
\end{equation*}
where $d_0 \colon V^{\kappa}(\mathfrak{b})\to V^{\widetilde{\kappa}}(\mathfrak{a})$ is an odd differential determined by
\begin{gather}
d_0(1)=0,\quad[d_0,\partial]=0,\label{afo2}\\
\begin{align}
d_0(f_{a,b}[-1])
&=\sum_{\col(b)\leq\col(c)<\col(a)}f_{c,b}[-1]\psi_{f_{a,c}}[-1]-\sum_{\substack{\col(b)<\col(c)\leq\col(a)}}\psi_{f_{c,b}}[-1]f_{a,c}[-1]\nonumber\\
&\quad+\alpha\psi_{f_{a,b}}[-2]+\delta(\col(a)>\col(-a)>\col(b))\psi_{f_{a,b}}[-2]\nonumber\\
&\quad+\delta(\col(a)\geq\col(-b)>\col(b))\psi_{f_{a,b}}[-2]\nonumber\\
&\quad+{(-1)}^{\widehat{p+2}+(\widehat{p}+\widehat{p+2})\cdot\widehat{i}}\psi_{f_{a+2n,b}}[-1]-{(-1)}^{\widehat{q}+(\widehat{q}+\widehat{q-2})\cdot\widehat{j}}\psi_{f_{a,b-2n}}[-1],
\end{align}\label{afo}
\end{gather}
where $i=\row(a),j=\row(b),p=\col(a),q=\col(b)$ and we assume that $\begin{cases}
f_u=0\text{ if }u\notin\mathfrak{b},\\
\psi_{f_v}=0\text{ if }v\notin\mathfrak{c}.
\end{cases}$..
\end{Definition}
Especially, we have
\begin{align}
d_0(f_{a,b}[-1])&={(-1)}^{\widehat{p+2}+(\widehat{p}+\widehat{p+2})\cdot\widehat{j}}\psi_{f_{a+2n,b}}[-1]-{(-1)}^{\widehat{p}+(\widehat{p}+\widehat{p-2})\cdot\widehat{i}}\psi_{f_{a,b-2n}}[-1]\label{yabu1}
\end{align}
provided that $\col(a)=\col(b)=p,\row(a)=j,\row(b)=i$ and
\begin{align}
&\quad d_0(f_{a,b}[-1])\nonumber\\
&=\sum_{\col(c)=\col(b)}f_{c,b}[-1]\psi_{f_{a,c}}[-1]-\sum_{\substack{\col(a)=\col(c)}}\psi_{f_{c,b}}[-1]f_{a,c}[-1]+\alpha\psi_{f_{a,b}}[-2]\nonumber\\
&\quad+\delta_{p,1}\psi_{f_{a,b}}[-2]+{(-1)}^{\widehat{p+2}+(\widehat{p}+\widehat{p+2})\cdot\widehat{j}}\psi_{f_{a+2n,b}}[-1]-{(-1)}^{\widehat{p-2}+(\widehat{p-2}+\widehat{p-4})\cdot\widehat{i}}\psi_{f_{a,b-2n}}[-1],\label{yabu2}
\end{align}
provided that $\col(a)=\col(b)+2=p,\row(a)=j,\row(b)=i$.

In the following theorem, we give two kinds of elements of $\mathcal{W}^k(\mathfrak{so}(nl),(l^n))$, which are in fact generators of $\mathcal{W}^k(\mathfrak{so}(nl),(l^n))$ (see Theorem~\ref{gener}).
\begin{Theorem}\label{GEN}
For $i,j\in I_n$, the rectangular $W$-algebra $\mathcal{W}^k(\mathfrak{so}(nl),(l^n))$ has the following elements;
\begin{align*}
W^{(1)}_{i,j}&=\sum_{\substack{\row(a)=j,\row(b)=i,\\\col(a)=\col(b)=p}}{(-1)}^{\widehat{p}\cdot(\widehat{j}+\widehat{i})}f_{a,b}[-1],
\end{align*}\\
\begin{align*}
W^{(2)}_{i,j}&=\alpha\sum_{\substack{\row(a)=j,\row(b)=i,\\\col(a)=\col(b)=p}}{(-1)}^{\widehat{p}\cdot(\widehat{j}+\widehat{i})}\dfrac{p}{2}f_{a,b}[-2]+\sum_{\substack{\row(a)=j,\row(b)=i,\\\col(a)=\col(b)+2=p}}{(-1)}^{\widehat{p}+\widehat{p}\cdot\widehat{j}+\widehat{p-2}\cdot\widehat{i}}f_{a,b}[-1]\\
&\quad+\sum_{\substack{\row(a_2)=j,\row(b_1)=i,\\p=\col(a_1)=\col(b_1)<\col(a_2)=\col(b_2)=q\\\row(a_1)=\row(b_2)=r}}{(-1)}^{(\widehat{r}+\widehat{i})\cdot\widehat{p}+(\widehat{j}+\widehat{r})\cdot\widehat{q}}f_{a_1,b_1}[-1]f_{a_2,b_2}[-1]\\
&\quad+\dfrac{1}{2}\sum_{\substack{\row(a)=j,\row(b)=i,\\\col(a)=\col(b)=p}}{(-1)}^{\widehat{p}+\widehat{p}\cdot(\widehat{j}+\widehat{i})}f_{a,b}[-2].
\end{align*}
\end{Theorem}
\begin{proof}
By Definition~\ref{T125}, it is enough to show that $d_0(W^{(r)}_{i,j})=0$. We only show the case when $r=2$. The case when $r=1$ is proven in a similar way.
By the definition of $W^{(2)}_{i,j}$, we have
\begin{align}
&\quad d_0(W^{(2)}_{i,j})\nonumber\\
&=d_0(\alpha\sum_{\substack{\row(a)=j,\row(b)=i,\\\col(a)=\col(b)=p}}{(-1)}^{\widehat{p}\cdot(\widehat{j}+\widehat{i})}\dfrac{p}{2}f_{a,b}[-2])+d_0(\sum_{\substack{\row(a)=j,\row(b)=i,\\\col(a)=\col(b)+2=p}}{(-1)}^{\widehat{p}+\widehat{p}\cdot\widehat{j}+\widehat{p-2}\cdot\widehat{i}}f_{a,b}[-1])\nonumber\\
&\quad+d_0(\sum_{\substack{\row(a_2)=j,\row(b_1)=i,\\p=\col(a_1)=\col(b_1)<\col(a_2)=\col(b_2)=q\\\row(a_1)=\row(b_2)=r}}{(-1)}^{(\widehat{r}+\widehat{i})\cdot\widehat{p}+(\widehat{j}+\widehat{r})\cdot\widehat{q}}f_{a_1,b_1}[-1]f_{a_2,b_2}[-1])\nonumber\\
&\quad+\dfrac{1}{2}d_0(\sum_{\substack{\row(a)=j,\row(b)=i,\\\col(a)=\col(b)=p}}{(-1)}^{\widehat{p}+\widehat{p}\cdot(\widehat{j}+\widehat{i})}f_{a,b}[-2]).\label{yab}
\end{align}
We compute each terms in the right hand side of \eqref{yab}. 
First, we compute the first term of the right hand side of \eqref{yab}. By \eqref{yabu1} and \eqref{afo2}, we can rewrite it as
\begin{align}
&\alpha\sum_{\substack{\row(a)=j,\row(b)=i,\\\col(a)=\col(b)=p}}{(-1)}^{\widehat{p+2}+\widehat{p+2}\cdot\widehat{j}+\widehat{p}\cdot\widehat{i}}\dfrac{p}{2}\psi_{f_{a+2n,b}}[-2]\nonumber\\
&\quad-\alpha\sum_{\substack{\row(a)=j,\row(b)=i,\\\col(a)=\col(b)=p}}{(-1)}^{\widehat{p}+\widehat{p-2}\cdot\widehat{i}+\widehat{p}\cdot\widehat{j}}\dfrac{p}{2}\psi_{f_{a,b-2n}}[-2].\label{yabu3.1}
\end{align}
Replacing $a$ and $b$ with $a+2n$ and $b+2n$, we can rewrite the second term of \eqref{yabu3.1} as
\begin{gather}
\alpha\sum_{\substack{\row(a)=j,\row(b)=i,\\\col(a)=\col(b)=p}}{(-1)}^{\widehat{p+2}+\widehat{p+2}\cdot\widehat{j}+\widehat{p}\cdot\widehat{i}}\dfrac{p+2}{2}\psi_{f_{a+2n,b}}[-2]\label{yabu3.2}
\end{gather}
Since $\col(a+2n)=\col(a)+2$, we find that
\begin{align}
&\quad\text{the first term of the right hand side of \eqref{yab}}\nonumber\\
&=-\alpha\sum_{\substack{\row(a)=j,\row(b)=i,\\\col(a)=\col(b)=p}}{(-1)}^{\widehat{p+2}+\widehat{p+2}\cdot\widehat{j}+\widehat{p}\cdot\widehat{i}}\psi_{f_{a+2n,b}}[-2]\label{yabu3}
\end{align}
by applying \eqref{yabu3.2} to \eqref{yabu3.1}.

Next, we compute the third term of the right hand side of \eqref{yab}. 
By \eqref{yabu1}, we can rewrite it as
\begin{align}
&\sum_{\substack{\row(a_2)=j,\row(b_1)=i,\\p=\col(a_1)=\col(b_1)<\col(a_2)=\col(b_2)=q\\\row(a_1)=\row(b_2)=r}}{(-1)}^{\beta+\widehat{p+2}+(\widehat{p}+\widehat{p+2})\cdot\widehat{j}}\psi_{f_{a_1+2n,b_1}}[-1]f_{a_2,b_2}[-1]\nonumber\\
&\quad-\sum_{\substack{\row(a_2)=j,\row(b_1)=i,\\p=\col(a_1)=\col(b_1)<\col(a_2)=\col(b_2)=q\\\row(a_1)=\row(b_2)=r}}{(-1)}^{\beta+\widehat{p}+(\widehat{p}+\widehat{p-2})\cdot\widehat{i}}\psi_{f_{a_1,b_1-2n}}[-1]f_{a_2,b_2}[-1]\nonumber\\
&\quad+\sum_{\substack{\row(a_2)=j,\row(b_1)=i,\\p=\col(a_1)=\col(b_1)<\col(a_2)=\col(b_2)=q\\\row(a_1)=\row(b_2)=r}}{(-1)}^{\beta+\widehat{q+2}+(\widehat{q}+\widehat{q+2})\cdot\widehat{j}}f_{a_1,b_1}[-1]\psi_{f_{a_2+2n,b_2}}[-1]\nonumber\\
&\quad-\sum_{\substack{\row(a_2)=j,\row(b_1)=i,\\p=\col(a_1)=\col(b_1)<\col(a_2)=\col(b_2)=q\\\row(a_1)=\row(b_2)=r}}{(-1)}^{\beta+\widehat{q}+(\widehat{q}+\widehat{q-2})\cdot\widehat{i}}f_{a_1,b_1}[-1]\psi_{f_{a_2,b_2-2n}}[-1],\label{yabu1156}
\end{align}
where $\beta={(\widehat{r}+\widehat{i})\cdot\widehat{p}+(\widehat{j}+\widehat{r})\cdot\widehat{q}}$. Let us set
\begin{align*}
\beta_1&=\beta+\widehat{p+2}+(\widehat{p}+\widehat{p+2})\cdot\widehat{j}\\
&=(\widehat{r}+\widehat{i})\cdot\widehat{p}+(\widehat{j}+\widehat{r})\cdot\widehat{q}+\widehat{p+2}+(\widehat{p}+\widehat{p+2})\cdot\widehat{r}.
\end{align*}
Then, we can rewrite the first term of \eqref{yabu1156} as
\begin{equation*}
\sum_{\substack{\row(a_2)=j,\row(b_1)=i,\\p=\col(a_1)=\col(b_1)<\col(a_2)=\col(b_2)=q\\\row(a_1)=\row(b_2)=r}}{(-1)}^{\beta_1}\psi_{f_{a_1+2n,b_1}}[-1]f_{a_2,b_2}[-1]
\end{equation*}
By setting $p'=p-2$, $a_1'=a_1-2n$, and $b_1=b_1-2n$, we can rewrite the second term of \eqref{yabu1156} as
\begin{align*}
&\quad\sum_{\substack{\row(a_2)=j,\row(b_1)=i,\\p'=\col(a_1')-2=\col(b_1')-2\\\qquad<\col(a_2)=\col(b_2)=q\\\row(a_1)=\row(b_2)=r}}{(-1)}^{{(\widehat{r}+\widehat{i})\cdot\widehat{p'+2}+(\widehat{j}+\widehat{r})\cdot\widehat{q}}+\widehat{p'+2}+(\widehat{p'+2}+\widehat{p'})\cdot\widehat{i}}\psi_{f_{a_1'+2n,b_1'}}[-1]f_{a_2,b_2}[-1]\\
&=\sum_{\substack{\row(a_2)=j,\row(b_1)=i,\\p'=\col(a_1')-2=\col(b_1')-2\\\qquad<\col(a_2)=\col(b_2)=q\\\row(a_1)=\row(b_2)=r}}{(-1)}^{{\widehat{r}\cdot\widehat{p'+2}+(\widehat{j}+\widehat{r})\cdot\widehat{q}}+\widehat{p'+2}+\widehat{p'}\cdot\widehat{i}}\psi_{f_{a_1'+2n,b_1'}}[-1]f_{a_2,b_2}[-1].
\end{align*}
Then, by a direct computation, we can rewrite the sum of the first two terms of \eqref{yabu1156} as
\begin{gather}
\sum_{\substack{\row(a_2)=j,\row(b_1)=i,\\\col(a_1)+2=\col(b_1)+2=\col(a_2)=\col(b_2)=q\\\row(a_1)=\row(b_2)}}{(-1)}^{\widehat{i}\cdot\widehat{q-2}+\widehat{j}\cdot\widehat{q}+\widehat{q}}\psi_{f_{a_1+2n,b_1}}[-1]f_{a_2,b_2}[-1]\label{yabu4.2}
\end{gather}
Similarly, we find that the third term of \eqref{yabu1156} is equal to
\begin{equation*}
\sum_{\substack{\row(a_2)=j,\row(b_1)=i,\\p=\col(a_1)=\col(b_1)<\col(a_2)-2=\col(b_2)-2=q\\\row(a_1)=\row(b_2)=r}}{(-1)}^{\beta_2}f_{a_1,b_1}[-1]\psi_{f_{a_2,b_2-2n}}[-1]
\end{equation*}
and
the 4-th term of \eqref{yabu1156} is equal to
\begin{equation*}
\sum_{\substack{\row(a_2)=j,\row(b_1)=i,\\p=\col(a_1)=\col(b_1)<\col(a_2)=\col(b_2)=q\\\row(a_1)=\row(b_2)=r}}{(-1)}^{\beta_2}f_{a_1,b_1}[-1]\psi_{f_{a_2,b_2-2n}}[-1],
\end{equation*}
where
\begin{align*}
\beta_2&=\beta+\widehat{q}+(\widehat{q}+\widehat{q-2})\cdot\widehat{i}\\
&=(\widehat{r}+\widehat{i})\cdot\widehat{p}+(\widehat{j}+\widehat{r})\cdot\widehat{q}+\widehat{q}+(\widehat{q}+\widehat{q-2})\cdot\widehat{r}.
\end{align*}
Then, we can rewrite the sum of the last two terms of \eqref{yabu1156} as
\begin{gather}
-\sum_{\substack{\row(a_2)=j,\row(b_1)=i,\\\col(a_1)=\col(b_1)=\col(a_2)-2=\col(b_2)-2=p\\\row(a_1)=\row(b_2)}}{(-1)}^{\widehat{i}\cdot\widehat{p}+\widehat{j}\cdot\widehat{p+2}+\widehat{p+2}}f_{a_1,b_1}[-1]\psi_{f_{a_2,b_2-2n}}[-1].\label{yabu4.3}
\end{gather}
Adding \eqref{yabu4.2} and \eqref{yabu4.3}, we have
\begin{align}
&\quad\text{the third term of the right hand side of \eqref{yab}}\nonumber\\
&=\sum_{\substack{\row(a_2)=j,\row(b_1)=i,\\\col(a_1)+2=\col(b_1)+2=\col(a_2)=\col(b_2)=q\\\row(a_1)=\row(b_2)}}{(-1)}^{\widehat{i}\cdot\widehat{q-2}+\widehat{j}\cdot\widehat{q}+\widehat{q}}\psi_{f_{a_1+2n,b_1}}[-1]f_{a_2,b_2}[-1]]\nonumber\\
&\quad-\sum_{\substack{\row(a_2)=j,\row(b_1)=i,\\\col(a_1)=\col(b_1)=\col(a_2)-2=\col(b_2)-2=p\\\row(a_1)=\row(b_2)}}{(-1)}^{\widehat{i}\cdot\widehat{p}+\widehat{j}\cdot\widehat{p+2}+\widehat{p+2}}f_{a_1,b_1}[-1]\psi_{f_{a_2,b_2-2n}}[-1].\label{yabu4}
\end{align}

Next, we compute the 4-th term of the right hand side of \eqref{yab}. By a direct computation, we obtain
\begin{align}
&\quad\dfrac{1}{2}d_0(\sum_{\substack{\row(a)=j,\row(b)=i,\\\col(a)=\col(b)=p}}{(-1)}^{\widehat{p}+\widehat{p}\cdot(\widehat{j}+\widehat{i})}f_{a,b}[-2])\nonumber\\
&=\dfrac{1}{2}\sum_{\substack{\row(a)=j,\row(b)=i,\\\col(a)=\col(b)=p}}{(-1)}^{\widehat{p+2}+(\widehat{p}+\widehat{p+2})\widehat{j}+\widehat{p}+\widehat{p}\cdot(\widehat{j}+\widehat{i})}\psi_{f_{a+2n,b}}[-2]\nonumber\\
&\quad-\dfrac{1}{2}\sum_{\substack{\row(a)=j,\row(b)=i,\\\col(a)=\col(b)=p}}{(-1)}^{\widehat{p}+(\widehat{p}+\widehat{p-2})\widehat{i}+\widehat{p}+\widehat{p}\cdot(\widehat{j}+\widehat{i})}\psi_{f_{a,b-2n}}[-2].\label{yabu4.198}
\end{align}
By a direct computation, we find that the second term of the right hand side of \eqref{yabu4.198} is equal to
\begin{align}
&-\dfrac{1}{2}\sum_{\substack{\row(a)=j,\row(b)=i,\\\col(a)=\col(b)=p}}{(-1)}^{\widehat{p+2}+(\widehat{p+2}+\widehat{p})\widehat{i}+\widehat{p+2}+\widehat{p+2}\cdot(\widehat{j}+\widehat{i})}\psi_{f_{a+2n,b}}[-2]\label{yabuyabu4.198}
\end{align}
Then, we have
\begin{align}
\text{the 4-th term of the right hand side of \eqref{yab}}
&=-{(-1)}^{\widehat{i}}\psi_{f_{n+j,-n+i}}[-2]\label{yabu923}
\end{align}
by applying \eqref{yabuyabu4.198} to \eqref{yabu4.198}.

Finally, we compute the second term of \eqref{yab}. 
By \eqref{yabu2}, we can rewrite the right hand side of the second term of \eqref{yab} as
\begin{align}
&\sum_{\substack{\row(a)=j,\row(b)=i,\\\col(a)=\col(b)+2=\col(c)+2=p}}{(-1)}^{\widehat{p}+\widehat{p}\cdot\widehat{j}+\widehat{p-2}\cdot\widehat{i}}f_{c,b}[-1]\psi_{f_{a,c}}[-1]\nonumber\\
&\quad-\sum_{\substack{\row(a)=j,\row(b)=i,\\\col(a)=\col(c)=\col(b)+2=p}}{(-1)}^{\widehat{p}+\widehat{p}\cdot\widehat{j}+\widehat{p-2}\cdot\widehat{i}}\psi_{f_{c,b}}[-1]f_{a,c}[-1]\nonumber\\
&\quad+\alpha\sum_{\substack{\row(a)=j,\row(b)=i,\\\col(a)=\col(b)+2=p}}{(-1)}^{\widehat{p}+\widehat{p}\cdot\widehat{j}+\widehat{p-2}\cdot\widehat{i}}\psi_{f_{a,b}}[-2]+{(-1)}^{\widehat{i}}\psi_{f_{n+j,-n+i}}[-2]\nonumber\\
&\quad+\sum_{\substack{\row(a)=j,\row(b)=i,\\\col(a)=\col(b)+2=p}}{(-1)}^{\widehat{p}+\widehat{p+2}+\widehat{p+2}\cdot\widehat{j}+\widehat{p-2}\cdot\widehat{i}}\psi_{f_{a+2n,b}}[-1]\nonumber\\
&\quad-\sum_{\substack{\row(a)=j,\row(b)=i,\\\col(a)=\col(b)+2=p}}{(-1)}^{\widehat{p+2}+\widehat{p}+\widehat{p-2}\cdot\widehat{i}+\widehat{p+2}\cdot\widehat{j}}\psi_{f_{a+2n,b}}[-1].\label{yabu2.1}
\end{align}
We can easily find that the sum of the last two terms of \eqref{yabu2.1} is equal to zero. We also find that the sum of \eqref{yabu4}(resp. \eqref{yabu3}, \eqref{yabu923}) and first and second terms (resp. third term, 4-th term) of \eqref{yabu2.1} is equal to zero.  
\end{proof}
\begin{Theorem}\label{gener}
Assume that $n\geq4$ and $\alpha\neq0$. The rectangular $W$-algebra $\mathcal{W}^k(\mathfrak{so}(nl),(l^n))$ is generated by $\{W^{(r)}_{i,j}\mid 1\leq i,j\leq n,r=1,2\}$.
\end{Theorem}
The proof of Theorem~\ref{gener} is given in the appendix. We prepare one lemma in order to prove the main theorem.
\begin{Lemma}\label{Lemma}
\textup{(1)}\ The following relations hold;
\begin{align*}
(W^{(1)}_{i,j})_{(0)}W^{(2)}_{v,w}
&=\delta_{i,w}W^{(2)}_{v,j}-\delta_{j,v}W^{(2)}_{i,w}+{(-1)}^{\widehat{i}+\widehat{j}}\delta_{i,-v}W^{(2)}_{-w,j}-{(-1)}^{\widehat{i}+\widehat{j}}\delta_{j,-w}W^{(2)}_{i,-v},\\
(W^{(1)}_{v,w})_{(1)}W^{(2)}_{i,j}
&=\dfrac{l-1}{2}\alpha(\delta_{j,v}W^{(1)}_{i,w}+\delta_{i,w}W^{(1)}_{v,j}+\delta_{-w,j}W^{(1)}_{i,-v}+\delta_{-v,i}W^{(1)}_{-w,j})\\
&\quad+\dfrac{1}{2}{(-1)}^{p(j)+p(i)}\delta_{v,-i}{W}^{(1)}_{-j,w}-\dfrac{1}{2}{(-1)}^{p(j)+p(v)}\delta_{w,i}{W}^{(1)}_{-j,-v}\\
&\quad-\dfrac{1}{2}\delta_{v,j}{W}^{(1)}_{i,w}+\dfrac{1}{2}{(-1)}^{p(v)+p(w)}\delta_{-w,j}{W}^{(1)}_{i,-v},\\
(W^{(1)}_{v,w})_{(s)}W^{(2)}_{i,j}&=0\ (s\geq2).
\end{align*}
\textup{(2)}\ We define a grading on $V^\kappa(\mathfrak{b})$ by setting $\text{deg}(x[-s])=j$ if $x\in\mathfrak{b}\cap\mathfrak{g}_{j}$. Then, we obtain
\begin{align*}
(W^{(2)}_{i,i})_{(1)}W^{(2)}_{j,j}
&=(1+\alpha\delta_{i,j}-\alpha{(-1)}^{\widehat{i}+\widehat{j}}\delta_{i,-j})(W^{(2)}_{i,i}+W^{(2)}_{j,j})+\text{higher terms}.
\end{align*}
\end{Lemma}
The proof is due to a direct computation. We omit it.
\section{Guay's affine Yangians}

In this section, we recall the definition of the Guay's affine Yangian and its coproduct.
\begin{Definition}\label{def-Yang}
Suppose that $n\geq3$ and set two $n\times n$-matrices $(a_{i,j})_{i,j\in I_n}$ and $(m_{i,j})_{i,j\in I_n}$ as
\begin{gather*}
a_{ij} =
	\begin{cases}
	2  &\text{if } i=j, \\
	-1 &\text{if } i=j \pm 2, \\
	        -1 &\text{if }(i,j)=(-n+1,n-1),(n-1,-n+1),\\
		0  &\text{otherwise,}
	\end{cases}\ 
	 m_{i,j}=
	\begin{cases}
	1&\text{if } i=j - 2,\\
		-1 &\text{if } i=j + 2,\\
	        1 &\text{if }(i,j)=(n-1,-n+1),\\
		-1 &\text{if }(i,j)=(-n+1,n-1),\\
		0  &\text{otherwise}.
	\end{cases}
\end{gather*}

The Guay's affine Yangian $Y_{\ve_1,\ve_2}(\widehat{\mathfrak{sl}}(n))$ is the associative algebra over $\mathbb{C}$ generated by $x_{i,r}^{+}, x_{i,r}^{-}$, $h_{i,r}$ $(i \in I_n, r \in \mathbb{Z}_{\geq 0})$ with parameters $\ve_1, \ve_2 \in \mathbb{C}$ subject to the following defining relations;
\begin{gather}
	[h_{i,r}, h_{j,s}] = 0, \label{eq1.1}\\
	[x_{i,r}^{+}, x_{j,s}^{-}] = \delta_{ij} h_{i, r+s}, \label{eq1.2}\\
	[h_{i,0}, x_{j,r}^{\pm}] = \pm a_{ij} x_{j,r}^{\pm},\label{eq1.3}\\
	[h_{i, r+1}, x_{j, s}^{\pm}] - [h_{i, r}, x_{j, s+1}^{\pm}] 
	= \pm a_{ij} \dfrac{\varepsilon_1 + \varepsilon_2}{2} \{h_{i, r}, x_{j, s}^{\pm}\} 
	- m_{ij} \dfrac{\varepsilon_1 - \varepsilon_2}{2} [h_{i, r}, x_{j, s}^{\pm}],\label{eq1.4}\\
	[x_{i, r+1}^{\pm}, x_{j, s}^{\pm}] - [x_{i, r}^{\pm}, x_{j, s+1}^{\pm}] 
	= \pm a_{ij}\dfrac{\varepsilon_1 + \varepsilon_2}{2} \{x_{i, r}^{\pm}, x_{j, s}^{\pm}\} 
	- m_{ij} \dfrac{\varepsilon_1 - \varepsilon_2}{2} [x_{i, r}^{\pm}, x_{j, s}^{\pm}],\label{eq1.5}\\
	\sum_{w \in \mathfrak{S}_{1 -a_{i,j}}}[x_{i,r_{w(1)}}^{\pm}, [x_{i,r_{w(2)}}^{\pm}, \dots, [x_{i,r_{w(1-a_{ij})}}^{\pm}, x_{j,s}^{\pm}]\dots]] = 0\  (i \neq j),\label{eq1.6}
\end{gather}
where we define $\{x,y\}$ by $xy+yx$ and $\mathfrak{S}_{1 -a_{ij}}$ is a permutation group of $\{1,2,\cdots,1-a_{i,j}\}$.
\end{Definition}
Here after, we sometimes denote $\ve_1+\ve_2$ by $\hbar$. In the case when $r=0$, we can rewrite \eqref{eq1.4} as
\begin{align}
[\tilde{h}_{i,1},x^\pm_{j,s}]&=\pm a_{i,j}(x^{\pm}_{j,s+1}-m_{i,j}\dfrac{\ve_1-\ve_2}{2}x^{\pm}_{j,s}),\label{eq1.8}
\end{align}
where $\tilde{h}_{i,1}=h_{i,1}-\dfrac{\hbar}{2}(h_{i,0})^2$. Then, by \eqref{eq1.8} and \eqref{eq1.2},
we find that $Y_{\ve_1,\ve_2}(\widehat{\mathfrak{sl}}(n))$ is generated by $\{h_{i,r},x^\pm_{i,r}\mid i\in I_n,r=0,1\}$ inductively. 
In fact, Guay have constructed a presentation of $Y_{\ve_1,\ve_2}(\widehat{\mathfrak{sl}}(n))$ whose generators are $\{h_{i,r},x^\pm_{i,r}\mid i\in I_n,r=0,1\}$ (see \cite{Gu1}, \cite{GNW}). This presentation is called the minimalistic presentation.
\begin{Remark}
The Guay's affine Yangian $Y_{\ve_1,\ve_2}(\widehat{\mathfrak{sl}}(n))$ was first introduced by Guay (see Definition 2.3 in \cite{Gu1}). Definition~\ref{def-Yang} is different from the definition given by Guay whose parameters are $\lambda,\beta$ and generators are $\{H_{i,r},X^\pm_{i,r}\mid0\leq i\leq n-1,r\in\mathbb{Z}_{\geq0}\}$. Then, the correspondence is given by
\begin{gather*}
\lambda=\hbar,\quad\beta=\dfrac{\hbar}{2}-\dfrac{n(\ve_1-\ve_2)}{4},\\
h_{i,0}=H_{i,0},\quad x^\pm_{i,0}=X^\pm_{i,0},\quad h_{i,1}=H_{i,1}-\Big(\dfrac{\hbar}{2}+(\dfrac{i}{2}-\dfrac{n}{4})(\ve_1-\ve_2)\Big)H_{i,0}.
\end{gather*}
\end{Remark}
We fix some notations about $\widehat{\mathfrak{sl}}(n)$. 
We set $\widehat{\mathfrak{sl}}(n)=\mathfrak{sl}(n)\otimes\mathbb{C}[t^{\pm1}]\oplus\mathbb{C}c$ as a Lie algebra whose commutator relations are given by
\begin{gather*}
[x\otimes t^s,y\otimes t^v]=[x,y]\otimes t^{s+v}+s\delta_{s+v,0}\text{tr}(xy)c,\\
\text{$c$ is a central element of $\widehat{\mathfrak{sl}}(n)$},
\end{gather*}
where tr is a trace of $\mathfrak{gl}(n)$. Let $K$ be the invariant inner product on $\widehat{\mathfrak{sl}}(n)$ defined by
\begin{equation*}
K(x\otimes t^s,y\otimes t^v)=\delta_{s+v,0}\tr(xy),\quad K(x\otimes t^s,c)=K(c,c)=0.
\end{equation*}
Let $\alpha_i$ be a simple root of $\widehat{\mathfrak{sl}}(n)$. We set an inner product on $\bigoplus_{i\in I_n}\mathbb{Z}\alpha_i$ by $(\alpha_i,\alpha_j)=a_{i,j}$.
We take Chevalley generators of $\widehat{\mathfrak{sl}}(n)=\mathfrak{sl}(n)\otimes\mathbb{C}[t^{\pm1}]\oplus\mathbb{C}c$ as follows;
\begin{gather*}
h_i=\begin{cases}
e_{n-1,n-1}-e_{-n+1,-n+1}+c&\text{ if }i=n-1,\\
e_{i,i}-e_{i+2,i+2}&\text{ if }i\neq n-1,
\end{cases}\\
x^+_i=\begin{cases}
e_{n-1,-n+1}t&\text{ if }i=n-1,\\
e_{i,i+2}&\text{ if }i\neq n-1,
\end{cases}\quad x^-_i=\begin{cases}
e_{-n+1,n-1}t^{-1}&\text{ if }i=n-1,\\
e_{i+2,i}&\text{ if }i\neq n-1.
\end{cases}
\end{gather*}
By an embedding $\widehat{\mathfrak{sl}}(n)\hookrightarrow Y_{\ve_1,\ve_2}(\widehat{\mathfrak{sl}}(n))$ determined by $h_i\mapsto h_{i,0}$ and $x^\pm_i\mapsto x^\pm_{i,0}$, we regard an element of $\widehat{\mathfrak{sl}}(n)$ as that of $Y_{\ve_1,\ve_2}(\widehat{\mathfrak{sl}}(n))$.

By using a minimalistic presentation, we construct the evaluation map for the Guay's affine Yangian. We set $\widehat{\mathfrak{gl}}(n)$ as a Lie algebra $\mathfrak{gl}(n)\otimes\mathbb{C}[t^{\pm 1}]\oplus\mathbb{C}c\oplus\mathbb{C}y$ whose commutator relations are determined by
\begin{gather*}
[e_{i,j}t^s,e_{p,q}t^u]=\delta_{j,p}e_{i,q}t^{s+u}-\delta_{i,q}e_{p,j}t^{s+u}+\delta_{j,p}\delta_{i,q}s\delta_{s+u,0}c+\delta_{i,j}\delta_{p,q}s\delta_{s+u,0}y,\\
c\text{ and }y\text{ are central elements}.
\end{gather*}
By setting the degree of $U(\widehat{\mathfrak{gl}}(n))$ by $\text{deg}(e_{i,j}t^s)=s$ and $\text{deg}(c)=\text{deg}(y)=0$, we denote the standard degreewise completion of $U(\widehat{\mathfrak{gl}}(n))/<c-n\ve_1,y-1>$ by $U(\widehat{\mathfrak{gl}}(n))_{{comp}}$ in the sense of \cite{MNT}.

\begin{Proposition}[\cite{Gu1} Section~6, \cite{K1} Theorem 3.8]
For any complex number $a$, there exists an algebra homomorphism 
\begin{equation*}
\ev_a \colon Y_{\ve_1,\ve_2}(\widehat{\mathfrak{sl}}(n)) \to U(\widehat{\mathfrak{gl}}(n))_{{comp}}
\end{equation*}
uniquely determined by 
\begin{gather*}
	\ev_a(x_{i,0}^{+}) = x_{i}^{+}, \quad \ev_a(x_{i,0}^{-}) = x_{i}^{-},\quad \ev_a(h_{i,0}) = h_{i},\\
	\ev_{a}(h_{i,1}) = \begin{cases}
		(a -n \varepsilon_1) h_{n-1} - \hbar e_{n-1,n-1} (e_{-n+1,-n+1}-c) \\
		\quad + \hbar \displaystyle\sum_{s \geq 0} \sum_{a\in I_n} \Big( e_{n-1,a}t^{-s} e_{a,n-1}t^s - e_{-n+1,a}t^{-s-1} e_{a,-n+1}t^{s+1} \Big) \text{ if $i = n-1$},\\
\\
		(a - i \varepsilon_1) h_{i} - \hbar e_{i,i}e_{i+2,i+2} \\
		\quad + \hbar \displaystyle\sum_{s \geq 0} \Big( \sum_{a\leq i} e_{i,a}t^{-s} e_{a,i}t^s + \displaystyle\sum_{a\geq i+2} e_{i,a}t^{-s-1} e_{a,i}t^{s+1} \\
		\qquad\qquad\quad - \displaystyle\sum_{a\leq i} e_{i+2,a}t^{-s} e_{a,i+2}t^s - \displaystyle\sum_{a\geq i+2} e_{i+2,a}t^{-s-1} e_{a,i+2}t^{s+1} \Big) \\
		\qquad\qquad\qquad\qquad\qquad\qquad\qquad\qquad\qquad\qquad\qquad\qquad\qquad\qquad \text{ if $i \neq n-1$}.
	\end{cases}
\end{gather*}
\end{Proposition}
Setting the grading on $Y_{\ve_1,\ve_2}(\widehat{\mathfrak{sl}}(n))$ determined by
\begin{equation*}
\text{deg}(h_{i,r})=0,\quad\text{deg}(x^+_{i,r})=\begin{cases}
1&\text{ if }i=n-1,\\
0&\text{ if }i\neq n-1,
\end{cases}\quad\text{deg}(x^-_{i,r})=\begin{cases}
-1&\text{ if }i=n-1,\\
0&\text{ if }i\neq n-1,
\end{cases}
\end{equation*}
we can define $\widetilde{Y}_{\ve_1,\ve_2}(\widehat{\mathfrak{sl}}(n))$ as the standard degreewise completion of $Y_{\ve_1,\ve_2}(\widehat{\mathfrak{sl}}(n))$ in the sense of \cite{MNT}. 
We also denote the standard degreewise completion of $Y_{\ve_1,\ve_2}(\widehat{\mathfrak{sl}}(n))\otimes Y_{\ve_1,\ve_2}(\widehat{\mathfrak{sl}}(n))$ by $Y_{\ve_1,\ve_2}(\widehat{\mathfrak{sl}}(n))\widehat{\otimes} Y_{\ve_1,\ve_2}(\widehat{\mathfrak{sl}}(n))$.
\begin{Theorem}[\cite{Gu1} Section~6, \cite{GNW} Theorem 4.9, \cite{KU} Theorem 7.1]\label{cop}
There exists an algebra homomorphism 
\begin{align*}
\Delta\colon Y_{\ve_1,\ve_2}(\widehat{\mathfrak{sl}}(n))\to Y_{\ve_1,\ve_2}(\widehat{\mathfrak{sl}}(n))\widehat{\otimes} Y_{\ve_1,\ve_2}(\widehat{\mathfrak{sl}}(n))
\end{align*}
determined by
\begin{gather*}
\Delta(x)=\square(x)\text{ for all }x\in\widehat{\mathfrak{sl}}(n),\\
\Delta(h_{i,1})=\square(h_{i,1})+\hbar h_i\otimes h_i-\hbar\sum_{\gamma\in\Delta^+_{\text{re}}}(\alpha_i,\gamma)x_{-\gamma}\otimes x_\gamma,
\end{gather*}
where $\square(x)=x\otimes1+1\otimes x$, $\Delta^+_{\text{re}}$ is a set of positive real root of $\widehat{\mathfrak{sl}}(n)$ and $x_\gamma$ is a root $\gamma$ element such that $K(x_\gamma,x_{-\gamma})=1$. Moreover, $\Delta$ satisfies the coassociativity.
\end{Theorem}
We note that we can naturally extend $\Delta$ to $\widetilde{\Delta}\colon\widetilde{Y}_{\ve_1,\ve_2}(\widehat{\mathfrak{sl}}(n))\to Y_{\ve_1,\ve_2}(\widehat{\mathfrak{sl}}(n))\widehat{\otimes}Y_{\ve_1,\ve_2}(\widehat{\mathfrak{sl}}(n))$ (see \cite{KU} Section~9). Here after, we denote $(\widetilde{\Delta}\otimes\id^{\otimes l-2})\circ\cdots(\widetilde{\Delta}\otimes\id)\circ\widetilde{\Delta}$ (resp. $(\square\otimes\id^{\otimes l-2})\circ\cdots(\square\otimes\id)\circ\square$) by $\widetilde{\Delta}^l$ (resp. $\square^l$).

\section{Twisted affine Yangians and rectangular $W$-algebras of type $D$}
Let us recall the Drinfeld $J$ presentation of the finite Yangian $Y_\hbar(\mathfrak{g})$. It is the original definition of Drinfeld (\cite{D1}).
\begin{Definition}[\cite{CP}, Section~12]\label{J}
Suppose that $\mathfrak{g}$ is a Kac-Moody Lie algebra of finite type. The Yangian $Y_\hbar(\mathfrak{g})$ is the associative algbra over $\mathbb{C}$ with generators $\{x,J(x)|x\in\mathfrak{g}\}$ subject to the following defining relations:
\begin{gather*}
xy-yx = [x,y] \text{ for all } x,y\in\mathfrak{g},\\
J(ax+by)=aJ(x)+bJ(y) \text{ for all } a,b\in\mathbb{C},\\
J([x,y]) = [x, J(y)],\\
[J(x), J([y,z])] + [J(z), J([x,y])] + [J(y), J([z,x])]= \hbar^2\sum_{a,b,c\in\mathtt{A}} ([x,\xi_a],[[y,\xi_b],[z,\xi_c]]) \{ \xi_a,\xi_b,\xi_c \},\\
\begin{align*}
&[[J(x),J(y)], [z,J(w)]]+[[J(z),J(w)],[x,J(y)]]\\
&\qquad\qquad\qquad=\hbar^2\sum_{a,b,c\in\mathtt{A}}\big(([x,\xi_a],[[y,\xi_b],[[z,w],\xi_c]])
+([z,\xi_a],[[w,\xi_b],[[x,y],\xi_c]]) \big)\{ \xi_a,\xi_b,J(\xi_c)\},
\end{align*}
\end{gather*}
where $(\ ,\ )$ is a non-zero invariant bilinear form, $\{ \xi_a \}_{a \in \mathtt{A}}$ is an orthonormal basis of $\mathfrak{g}$ and $\{ \xi_a,\xi_b,\xi_c \} = \dfrac{1}{24} \sum_{\pi\in G} \xi_{\pi(a)} \xi_{\pi(b)} \xi_{\pi(c)}$, $G$ being the group of permutations of $\{ a,b,c\}$.
By Definition~\ref{J},we note that there exists an isomorphism of $\chi_\hbar\colon Y_\hbar(\mathfrak{g})\to Y_{-\hbar}(\mathfrak{g})$ determined by $x\mapsto x$ and $J(x)\mapsto J(x)$.
\end{Definition}
Belliard and Regelskis (\cite{BR1}) gave generators of twisted Yangians in the words of the Drinfeld $J$ presentation. 
\begin{Theorem}[\cite{BR1}, Theorem~5.5]\label{JJ}
Let $\big(\mathfrak{g},\mathfrak{g}^\theta\big)$ be a symmetric pair of a finite-dimensional simple complex Lie algebra $\mathfrak{g}$ of $\text{rank}(\mathfrak{g})\geq2$ with respect to the involution $\theta$, such that $\mathfrak{g}^\theta$ is the positive eigenspace of $\theta$. Let $\{X_a\}$ (resp. $\{Y_p\}$) be a basis of $\mathfrak{g}^\theta$ (resp. $\{x\in\mathfrak{g}\mid\theta(x)=-x\}$). We decompose the Cartan element of $\mathfrak{g}$ into $C_{\mathfrak{k}}+C_{\mathfrak{m}}$, where $C_{\mathfrak{k}}$ (resp. $C_{\mathfrak{m}}$) is an element of $U(\mathfrak{k})$ (resp. $\mathbb{C}[\mathfrak{m}]$).
Then, the twisted yangian $T_\hbar(\mathfrak{g},\mathfrak{g}^\theta)$ is isomorphic to the subalgebra of $Y_h(\mathfrak{g})$ generated by $\{X_a,B(Y_p)\}$, where
\begin{equation*}
B(Y_p)=J(Y_p)+\dfrac{\hbar}{4}\Big[Y_p,C_{\mathfrak{k}}\Big].
\end{equation*}
\end{Theorem}
Belliard and Regelskis also gave the Drinfeld $J$ presentation of twisted Yangians whose generators are $\{X_a,B(Y_p)\}$. Its defining relations contain the relation $[X_a,B(Y_p)]=B([X_a,Y_p])$. By Theorem~\ref{JJ}, we can realize $T_\hbar(\mathfrak{g},\mathfrak{g}^\theta)$ as a subalgebra of $Y_{-\hbar}(\mathfrak{g})$ via $\chi_\hbar$.

In finite setting, Brown \cite{Bro} constructed a surjective homomorphism from the twisted Yangians to the finite rectangular $W$-algebras. Especially, in the case when $n,l$ is a he has constructed a surjective homorphism from $T_h(\mathfrak{sl}(n),\mathfrak{sp}(n))$ to the rectangular finite $W$-algebra associated with $\mathfrak{so}(nl)$ and a nilpotent element whose Jordan form corresponds to the partition $(l^n)$.
We construct the affine analogue of this homomorphism.

There exists the following symmetric pair decomposition of $\mathfrak{sl}(n)$;
\begin{gather*}
\mathfrak{sl}(n)=<e_{i,j}-{(-1)}^{\widehat{i}+\widehat{j}}e_{-j,-i}\mid i,j\in I_n>\oplus\Big(<e_{i,j}-{(-1)}^{\widehat{i}+\widehat{j}}e_{-j,-i}\mid i,j\in I_n>\cap\mathfrak{sl}(n)\Big),
\end{gather*}
where $<A_i\mid i\in B>$ is a $\mathbb{C}$-vector space spanned by $\{A_i\}_{i\in B}$.
Let $H_i$ be $e_{i,i}-e_{i+2,i+2}\in\mathfrak{sl}(n)$. We set
\begin{align*}
\mathfrak{k}&=<e_{i,j}-{(-1)}^{\widehat{i}+\widehat{j}}e_{-j,-i}\mid i,j\in I_n>,\\
\mathfrak{m}&=<e_{i,j}+{(-1)}^{\widehat{i}+\widehat{j}}e_{-j,-i}\mid i,j\in I_n>\cap\mathfrak{sl}(n)\\
&=<H_i-H_{-i-2}\mid i,j\in I_n\setminus \{n-1\}>\oplus<e_{i,j}+{(-1)}^{\widehat{i}+\widehat{j}}e_{-j,-i}\mid i,j\in I_n,i\neq j>.
\end{align*}
We note that $\mathfrak{k}$ is isomorphic to $\mathfrak{sp}(n)$. Moreover, we have the following lemma.
\begin{Lemma}\label{mmm}
Any element of $\mathfrak{m}$ can be written as $[H_i-H_{-i-2},x]$ or $[[H_i-H_{-i-2},x],y]$ for some $x,y\in\mathfrak{k}$ and $i\in I_n\setminus \{n-1\}$.
\end{Lemma}
\begin{proof}
By a direct computation, we obtain
\begin{align*}
&\quad[H_i-H_{-i-2},e_{i,j}-{(-1)}^{\widehat{i}+\widehat{j}}e_{-j,-i}]\\
&=(1+\delta_{j,i+2}+\delta_{i,-1}+\delta_{j,-i-2})(e_{i,j}+{(-1)}^{\widehat{i}+\widehat{j}}e_{-j,-i})
\end{align*}
for all $i\neq\pm j$. By a direct computation, we have
\begin{align*}
&\quad[e_{i,j}+{(-1)}^{\widehat{i}+\widehat{j}}e_{-j,-i},e_{j,i}-{(-1)}^{\widehat{i}+\widehat{j}}e_{-i,-j}]\\
&=e_{i,i}+e_{-i,-i}-(e_{j,j}+e_{-j,-j})
\end{align*}
for all $i\neq\pm j$.
This completes the proof.
\end{proof}
By Lemma~\ref{mmm}, we find that the twisted Yangian $T_h(\mathfrak{sl}(n),\mathfrak{sp}(n))$ is generated by $\mathfrak{k}$ and $\{B(H_i)-B(H_{-i-2})\mid i\in I_n\setminus \{n-1\}\}$. By Theorem~\ref{JJ}, we can rewrite $B(H_i-H_{-i-2})$ as
\begin{align*}
&J(H_i-H_{-i-2})+\dfrac{\hbar}{8}\Big[H_i-H_{-i-2}, \sum_{\substack{u>v}}(e_{u,v}-{(-1)}^{\widehat{u}+\widehat{v}}e_{-v,-u})(e_{v,u}-{(-1)}^{\widehat{u}+\widehat{v}}e_{-u,-v})\Big]
\end{align*}
for all $i\in I_n\setminus \{n-1\}$. 

In a similar way to Theorem~\ref{JJ}, we define the twisted affine Yangian of type $C$. We have a decomposition $\widehat{\mathfrak{sl}}(n)=\widehat{\mathfrak{k}}\oplus\mathfrak{m}\otimes\mathbb{C}[t^{\pm1}]$. 
Similarly to the finite case, we note that $\widehat{\mathfrak{k}}$ is isomorphic to $\widehat{\mathfrak{sp}}(n)$ and $\mathfrak{m}\otimes\mathbb{C}[t^{\pm1}]=[h_i-h_{-i-2},\widehat{\mathfrak{k}}]+[[h_i-h_{-i-2},\widehat{\mathfrak{k}}],\widehat{\mathfrak{k}}]$.

By the similar formula in Section~3 of \cite{GNW}, we can define $J(h_i)$ as an element of $\widetilde{Y}_{\ve_1,\ve_2}(\widehat{\mathfrak{sl}}(n))$;
\begin{gather*}
J(h_i)=h_{i,1}+\dfrac{\hbar}{2}\sum_{\gamma\in\Delta^+_{\text{re}}}(\alpha_i,\gamma)x_{-\gamma}x_\gamma-\dfrac{\hbar}{2}h_i^2,
\end{gather*}
where $\Delta^+_{\text{re}}$ is a set of positive real root of $\widehat{\mathfrak{sl}}(n)$ and $x_\gamma$ is a root $\gamma$ element such that $(x_\gamma,x_{-\gamma})=1$. By the definition of $J(h_i)$ and Theorem~\ref{cop}, we obtain
\begin{align}
\widetilde{\Delta}(J(h_i))&=\square(J(h_i))+\dfrac{\hbar}{2}\sum_{\gamma\in\Delta^+_{\text{re}}}(\alpha_i,\gamma)(x_{\gamma}\otimes x_{-\gamma}-x_{-\gamma}\otimes x_\gamma)\nonumber\\
&=\square(J(h_i))+\dfrac{\hbar}{2}\sum_{\gamma\in\Delta_{\text{re}}}[h_i,x_{\gamma}]\otimes x_{-\gamma},\label{align6598}
\end{align}
where $\Delta_{\text{re}}$ is a set of real roots of $\widehat{\mathfrak{sl}}(n)$.
\begin{Definition}
For all $i\in I_n\setminus \{n-1\}$, let us set $B(h_i-h_{-i-2})$ as
\begin{align}
&J(h_i-h_{-i-2})+\dfrac{\hbar}{8}[\sum_{\substack{u<v\\m\geq1}}(e_{u,v}-{(-1)}^{\widehat{u}+\widehat{v}}e_{-v,-u})t^{-m}(e_{v,u}-{(-1)}^{\widehat{u}+\widehat{v}}e_{-u,-v})t^m,h_i-h_{-i-2}]\nonumber\\
&\qquad\qquad\qquad\qquad+\dfrac{\hbar}{8}[\sum_{\substack{u>v\\m\geq0}}(e_{u,v}-{(-1)}^{\widehat{u}+\widehat{v}}e_{-v,-u})t^{-m}(e_{v,u}-{(-1)}^{\widehat{u}+\widehat{v}}e_{-u,-v})t^m,h_i-h_{-i-2}].\label{JJJ2}
\end{align}
We define $TY_{\ve_1,\ve_2}(\widehat{\mathfrak{sp}}(n))$ as a subalgebra of $\widetilde{Y}_{\ve_1,\ve_2}(\widehat{\mathfrak{sl}}(n))$ topologically generated by $\widehat{\mathfrak{k}}$ and $\{B(h_i-h_{-i-2})\mid i\in I_n\setminus \{n-1\}\}$.
\end{Definition}
By the definition of $TY_{\ve_1,\ve_2}(\widehat{\mathfrak{sp}}(n))$, the universal enveloping algebra of $\widehat{\mathfrak{k}}$ can be embedded into $TY_{\ve_1,\ve_2}(\widehat{\mathfrak{sp}}(n))$.
\begin{Proposition}\label{Prop}
The restriction of $\widetilde{\Delta}\colon\widetilde{Y}_{\ve_1,\ve_2}(\widehat{\mathfrak{sl}}(n))\to Y_{\ve_1,\ve_2}(\widehat{\mathfrak{sl}}(n))\widehat{\otimes}Y_{\ve_1,\ve_2}(\widehat{\mathfrak{sl}}(n))$ gives a coideal structure to $TY_{\ve_1,\ve_2}(\widehat{\mathfrak{so}}(2n))$. That is, we have
\begin{align*}
\widetilde{\Delta}(TY_{\ve_1,\ve_2}(\widehat{\mathfrak{sp}}(n)))\subset TY_{\ve_1,\ve_2}(\widehat{\mathfrak{sp}}(n))\widehat{\otimes}Y_{\ve_1,\ve_2}(\widehat{\mathfrak{sl}}(n)),
\end{align*}
where the completed tensor product $Y_{\ve_1,\ve_2}(\widehat{\mathfrak{sl}}(n))\widehat{\otimes}TY_{\ve_1,\ve_2}(\widehat{\mathfrak{sp}}(n))$ is defined in the same way as $Y_{\ve_1,\ve_2}(\widehat{\mathfrak{sl}}(n))\widehat{\otimes}Y_{\ve_1,\ve_2}(\widehat{\mathfrak{sl}}(n))$.
\end{Proposition}
\begin{proof}
It is enough to show that $\widetilde{\Delta}(B(h_i-h_{-i-2}))\subset TY_{\ve_1,\ve_2}(\widehat{\mathfrak{sp}}(n))\widehat{\otimes}Y_{\ve_1,\ve_2}(\widehat{\mathfrak{sl}}(n))$. 
By the definition of $B(h_i-h_{-i-2})$ and \eqref{align6598}, we find that 
\begin{equation*}
\widetilde{\Delta}(B(h_i-h_{-i-2}))=\square(B(h_i-h_{-i-2}))+C_i+C_{-i}-C_{i+2}-C_{-i-2}+D_i+D_{-i}-D_{i+2}-D_{-i-2},
\end{equation*}
where
\begin{align*}
C_i&=\dfrac{\hbar}{2}\sum_{\gamma\in\Delta_{\text{re}}}[e_{i,i},x_{\gamma}]\otimes x_{-\gamma},\\
D_i&=\dfrac{\hbar}{8}[\sum_{\substack{u\neq v\\m\in\mathbb{Z}}}(e_{u,v}-{(-1)}^{\widehat{u}+\widehat{v}}e_{-v,-u})t^{-m}\otimes(e_{v,u}-{(-1)}^{\widehat{u}+\widehat{v}}e_{-u,-v})t^m,\square e_{i,i}].
\end{align*}
By a direct computation, we obtain
\begin{align}
C_i&=-\dfrac{\hbar}{2}\sum_{\substack{u\neq i\\s\in\mathbb{Z}}}e_{u,i}t^{-s}\otimes e_{i,u}t^{s}+\dfrac{\hbar}{2}\sum_{\substack{u\neq i\\s\in\mathbb{Z}}}e_{i,u}t^{s+1}\otimes e_{u,i}t^{-s-1}.\label{align9}
\end{align}
By a direct computation, we also obtain
\begin{align}
D_i&=-\dfrac{\hbar}{4}\sum_{\substack{v\neq i\\s\in\mathbb{Z}}}{(-1)}^{\widehat{v}+\widehat{i}}e_{-v,-i}t^{-s}\otimes e_{v,i}t^s-\dfrac{\hbar}{4}\sum_{\substack{v\neq -i\\s\in\mathbb{Z}}}{(-1)}^{\widehat{v}+\widehat{-i}}e_{-v,i}t^{-s}\otimes e_{v,-i}t^s\nonumber\\
&\quad+\dfrac{\hbar}{4}\sum_{\substack{u\neq i\\s\in\mathbb{Z}}}{(-1)}^{\widehat{u}+\widehat{i}}e_{-i,-u}t^{-s}\otimes e_{i,u}t^s+\dfrac{\hbar}{4}\sum_{\substack{u\neq -i\\s\in\mathbb{Z}}}{(-1)}^{\widehat{u}+\widehat{-i}}e_{i,-u}t^{-s}\otimes e_{-i,u}t^s.
\end{align}
By setting
\begin{align}
F_i&=-\dfrac{\hbar}{2}\sum_{\substack{v\neq i\\s\in\mathbb{Z}}}{(-1)}^{\widehat{v}+\widehat{i}}e_{-v,-i}t^{-s}\otimes e_{v,i}t^s+\dfrac{\hbar}{2}\sum_{\substack{u\neq i\\s\in\mathbb{Z}}}{(-1)}^{\widehat{u}+\widehat{i}}e_{-i,-u}t^{-s}\otimes e_{i,u}t^s,\label{align15}
\end{align}
we obtain $D_i+D_{-i}=F_i+F_{-i}$ by a direct computation. We denote the $a$-th term of the right hand side of \eqref{align9} (resp. \eqref{align15}) by $C_{i,a}$ (resp. $F_{i,a}$). Then, we find that
\begin{gather*}
C_{i,1}+F_{i,2}=-\dfrac{\hbar}{2}\sum_{\substack{u\neq i\\s\in\mathbb{Z}}}(e_{u,i}t^{-s}-{(-1)}^{\widehat{u}+\widehat{i}}e_{-i,-u}t^{-s})\otimes e_{i,u}t^{s},\\
C_{i,2}+F_{i,1}=\dfrac{\hbar}{2}\sum_{\substack{u\neq i\\s\in\mathbb{Z}}}(e_{i,u}t^{s}-{(-1)}^{\widehat{u}+\widehat{i}}e_{-u,-i}t^s)\otimes e_{u,i}t^{-s}.
\end{gather*}
Since $C_{i,1}+F_{i,2}$ and $C_{i,2}+F_{i,1}$ are contained in $TY_{\ve_1,\ve_2}(\widehat{\mathfrak{sp}}(n))\widehat{\otimes}Y_{\ve_1,\ve_2}(\widehat{\mathfrak{sl}}(n))$, we find that
\begin{equation*}
\widetilde{\Delta}(B(h_i-h_{-i-2}))=\square(B(h_i-h_{-i-2}))+C_i+C_{-i}-C_{i+2}-C_{-i-2}+F_i+F_{-i}-F_{i+2}-F_{-i-2}
\end{equation*}
is contained in $TY_{\ve_1,\ve_2}(\widehat{\mathfrak{sp}}(n))\widehat{\otimes}Y_{\ve_1,\ve_2}(\widehat{\mathfrak{sl}}(n))$. 
\end{proof}
Next, we construct a homomorphsim from the twisted affine Yangian to the universal enveloping algebra of the rectangular $W$-algebra of type $D$.
We recall the definition of the universal enveloping algebra of a vertex algebra.
For a vertex algebra $V$, let $L(V)$ be the Borchards Lie algebra, that is,
\begin{align}
 L(V)=V{\otimes}\mathbb{C}[t,t^{-1}]/\text{Im}(\partial\otimes\id +\id\otimes\frac{d}{d t})\label{844},
\end{align}
where the commutation relation is given by
\begin{align*}
 [ut^a,vt^b]=\sum_{r\geq 0}\begin{pmatrix} a\\r\end{pmatrix}(u_{(r)}v)t^{a+b-r}
\end{align*}
for all $u,v\in V$ and $a,b\in \mathbb{Z}$. 
\begin{Definition}[Frenkel-Zhu~\cite{FZ}, Matsuo-Nagatomo-Tsuchiya~\cite{MNT}]\label{Defi}
We set $\mathcal{U}(V)$ as the quotient algebra of the standard degreewise completion of the universal enveloping algebra of $L(V)$ by the completion of the two-sided ideal generated by
\begin{gather}
(u_{(a)}v)t^b-\sum_{i\geq 0}
\begin{pmatrix}
 a\\i
\end{pmatrix}
(-1)^i(ut^{a-i}vt^{b+i}-{(-1)}^{p(u)p(v)}(-1)^avt^{a+b-i}ut^{i}),\label{241}\\
|0\rangle t^{-1}-1.
\end{gather}
We call $\mathcal{U}(V)$ the universal enveloping algebra of $V$.
\end{Definition}
Let $l'$ be $\dfrac{l}{2}$. There exists an isomorphism $\Psi\colon\mathfrak{g}_0\to\mathfrak{gl}_n^{\otimes l'}=\mathfrak{L}$ determined by
\begin{gather*}
\Psi(f_{a,b})=e_{\row(a),\row(b)}[-1]\text{ if }\col(a)=\col(b)>0.
\end{gather*}
We denote $1^{\otimes r-1}\otimes e_{i,j}\otimes1^{\otimes l'-r}$ by $e_{i,j}^{(r)}$. The projection $\mathfrak{so}(nl)\to\mathfrak{g}_0$ induces the Miura transformation (\cite{KW})
\begin{align*}
\mu_D\colon \mathcal{W}^k(\mathfrak{so}(nl),(l^n))\to V^\Gamma(\mathfrak{L}),
\end{align*}
where
\begin{align*}
\Gamma(e_{a,b}^{(r_1)},e_{c,d}^{(r_2)})=\delta_{a,d}\delta_{b,c}\delta_{r_1,r_2}\alpha+\delta_{a,b}\delta_{c,d}\delta_{r_1,r_2}.
\end{align*}
The Miura transformation is injective (see \cite{F}, \cite{A3}). The Miura transformation also induces the homomorphism
\begin{equation*}
\widetilde{\mu}_D\colon\mathcal{U}(\mathcal{W}^k(\mathfrak{so}(nl),(l^n)))\to U(\widehat{\mathfrak{gl}}(n))^{\otimes l}_{\text{comp}},
\end{equation*}
where $U(\widehat{\mathfrak{gl}}(n))^{\otimes l}_{\text{comp}}$ is the standard degreewise completion of $U(\widehat{\mathfrak{gl}}(n))^{\otimes l}$ in the sense of \cite{MNT}. By using the PBW theorem of the universal enveloping algebra of the vertex algebra (\cite{A} Section 3.14), we note that $\widetilde{\mu}_D$ is injective.
By the definition of $\mu_D$, we have
\begin{align*}
\widetilde{\mu}_D(W^{(1)}_{i,j}t^s)&=\sum_{1\leq r\leq l'}(e^{(r)}_{j,i}-{(-1)}^{\widehat{i}+\widehat{j}}e^{(r)}_{-i,-j})t^s
\end{align*}
and
\begin{align}
\widetilde{\mu}_D(W^{(2)}_{i,i}t)&=\sum_{\substack{1\leq r_1<r_2\leq l'\\s\in\mathbb{Z}}}e^{(r_1)}_{u,i}t^{-s}e^{(r_2)}_{i,u}t^s+\sum_{\substack{1\leq r_1<r_2\leq l'\\s\in\mathbb{Z}}}e^{(r_1)}_{u,-i}t^{-s}e^{(r_2)}_{-i,u}t^s\nonumber\\
&\quad-\sum_{\substack{1\leq r_1,r_2\leq l'\\u<i,s\geq0}}\limits{(-1)}^{\widehat{u}+\widehat{i}}e^{(r_1)}_{-i,-u}t^{-s}e^{(r_2)}_{i,u}t^s-\sum_{\substack{r_1,r_2\\u<i\\s\geq1}}\limits{(-1)}^{\widehat{u}+\widehat{i}}e^{(r_1)}_{i,u}t^{-s}e^{(r_2)}_{-i,-u}t^s\nonumber\\
&\quad-\sum_{\substack{1\leq r_1,r_2\leq l'\\u>i,s\geq0}}\limits{(-1)}^{\widehat{u}+\widehat{i}}e^{(r_1)}_{-i,-u}t^{-s}e^{(r_2)}_{i,u}t^s-\sum_{\substack{1\leq r_1,r_2\leq l'\\u>i,s\geq1}}\limits{(-1)}^{\widehat{u}+\widehat{i}}e^{(r_1)}_{i,u}t^{-s}e^{(r_2)}_{-i,-u}t^s\nonumber\\
&\quad-\sum_{\substack{1\leq r_1,r_2\leq l'\\s\geq0}}\limits e^{(r_1)}_{-i,-i}t^{-s}e^{(r_2)}_{i,i}t^s-\sum_{\substack{1\leq r_1,r_2\leq l'\\s\geq1}}\limits e^{(r_1)}_{i,i}t^{-s}e^{(r_2)}_{-i,-i}t^s\nonumber\\
&\quad-\sum_{\substack{1\leq r\leq l'}}\dfrac{(2r+1)\alpha+1}{2}e_{i,i}^{(r)}-\sum_{\substack{1\leq r\leq l'}}\dfrac{(2r+1)\alpha+1}{2}e_{-i,-i}^{(r)}.\label{align5}
\end{align}
\begin{Theorem}\label{main thm}
For $n\geq4$ and $l\geq2$, there exists a homomorphism 
\begin{equation*}
\Phi\colon TY(\widehat{\mathfrak{so}}(n))\to \mathcal{U}(\mathcal{W}^k(\mathfrak{so}(nl),(l^n)))
\end{equation*}
defined by $\tilde{\mu}_D\circ\Phi=(\bigotimes_{r=1}^l\limits\ev_{\xi_r})\circ\widetilde{\Delta}^l$,
where $\xi_r=\dfrac{(2r+1)\alpha+1}{2}\hbar$.
\end{Theorem}
\begin{proof}
It is enough to show that 
\begin{equation*}
((\bigotimes_{r=1}^l\limits\ev_{\xi_r})\circ\widetilde{\Delta}^l)(x)\in\text{Im}(\tilde{\mu}_D)\text{ for all }x\in\widehat{\mathfrak{k}}\text{ or }x=B(h_i-h_{-i-2}).
\end{equation*}
First, we show that $((\bigotimes_{r=1}^l\limits\ev_{\xi_r})\circ\widetilde{\Delta}^l)(\widehat{\mathfrak{k}})\subset\text{Im}(\tilde{\mu}_D)$. By a direct computation, we obtain
\begin{align*}
&\quad((\bigotimes_{r=1}^l\limits\ev_{\xi_r})\circ\widetilde{\Delta}^l)\big((e_{i,j}-{(-1)}^{\hat{i}+\hat{j}}e_{-j,-i})t^s\big)\\
&=(\bigotimes_{r=1}^l\limits\ev_{\xi_r})\big(\square^l(e_{i,j}-{(-1)}^{\hat{i}+\hat{j}}e_{-j,-i})t^s\big)\\
&=\displaystyle\sum_{1\leq r\leq l}\limits(e_{i,j}^{[r]}-{(-1)}^{\hat{i}+\hat{j}}e_{-j,-i}^{[r]})t^s\\
&=\tilde{\mu}_D(W^{(1)}_{j,i}t^s).
\end{align*}
Thus, we obtain $((\bigotimes_{r=1}^l\limits\ev_{\xi_r})\circ\widetilde{\Delta}^l)(\widehat{\mathfrak{k}})\subset\text{Im}(\tilde{\mu}_D)$.

Next, we show that $((\bigotimes_{r=1}^l\limits\ev_{\xi_r})\circ\widetilde{\Delta}^l)(B(h_i-h_{-i-2}))\in\text{Im}(\tilde{\mu}_D)$.
By the proof of Proposition~\ref{Prop}, we obtain
\begin{align*}
&\quad((\bigotimes_{r=1}^l\limits\ev_{\xi_r})\circ\widetilde{\Delta}^l)(B(h_i-h_{-i-2}))\\
&=\square^l(\ev_0(B(h_i-h_{-i-2})))+C_i^l+C_{-i}^l-C_{i+2}^l-C_{-i-2}^l+F_i^l+F_{-i}^l-F_{i+2}^l-F_{-i-2}^l+G_i-G_{i+2},
\end{align*}
where
\begin{align}
C_i^l&=-\dfrac{\hbar}{2}\sum_{\substack{1\leq r_1<r_2\leq l'\\u\neq i,s\in\mathbb{Z}}}e_{u,i}^{(r_1)}t^{-s} e_{i,u}^{(r_2)}t^{s}+\dfrac{\hbar}{2}\sum_{\substack{1\leq r_1<r_2\leq l'\\u\neq i,s\in\mathbb{Z}}}e_{i,u}^{(r_1)}t^{s} e_{u,i}^{(r_2)}t^{-s},\label{3.5.6}\\
F_i^l&=-\dfrac{\hbar}{2}\sum_{\substack{1\leq r_1<r_2\leq l'\\v\neq i,s\in\mathbb{Z}}}{(-1)}^{\widehat{v}+\widehat{i}}e^{(r_1)}_{-v,-i}t^{-s}e^{(r_2)}_{v,i}t^s+\dfrac{\hbar}{2}\sum_{\substack{1\leq r_1<r_2\leq l'\\u\neq i,s\in\mathbb{Z}}}{(-1)}^{\widehat{u}+\widehat{i}}e^{(r_1)}_{-i,-u}t^{-s}e^{(r_2)}_{i,u}t^s,\label{3.5.8}\\
G_i&=\hbar\sum_{\substack{1\leq r\leq l'}}\dfrac{(2r+1)\alpha+1}{2}e_{i,i}^{(r)}+\hbar\sum_{\substack{1\leq r\leq l'}}\dfrac{(2r+1)\alpha+1}{2}e_{-i,-i}^{(r)}.\label{3.5.9.1}
\end{align}
By the definition of $C_i^l$ and $F_i^l$, we have
\begin{align}
C_{-i}^l&=-\dfrac{\hbar}{2}\sum_{\substack{1\leq r_1<r_2\leq l'\\u\neq i,s\in\mathbb{Z}}}e_{-u,-i}^{(r_1)}t^{-s}e_{-i,-u}^{(r_2)}t^{s}+\dfrac{\hbar}{2}\sum_{\substack{1\leq r_1<r_2\leq l'\\u\neq i,s\in\mathbb{Z}}}e_{-i,-u}^{(r_1)}t^se_{-u,-i}^{(r_2)}t^{-s},\label{3.5.7}\\
F_{-i}^l&=-\dfrac{\hbar}{2}\sum_{\substack{1\leq r_1<r_2\leq l'\\v\neq i,s\in\mathbb{Z}}}{(-1)}^{\widehat{v}+\widehat{i}}e^{(r_1)}_{v,i}t^{-s}\otimes e^{(r_2)}_{v,i}t^s+\dfrac{\hbar}{2}\sum_{\substack{1\leq r_1<r_2\leq l'\\u\neq i,s\in\mathbb{Z}}}{(-1)}^{\widehat{u}+\widehat{i}}e^{(r_1)}_{i,u}t^{-s}\otimes e^{(r_2)}_{-i,-u}t^s.\label{3.5.9}
\end{align}
For a while, we compute $\ev_0(B(h_i-h_{-i-2}))$.
By a direct computation, we obtain 
\begin{align*}
\ev_0(J(h_i))&=A_i-A_{i+2},
\end{align*}
where
\begin{align}
A_i&=-\dfrac{\hbar}{2}e_{i,i}^2+\dfrac{\hbar}{2}\sum_{\substack{u>i\\s\geq0}}e_{u,i}t^{-s}e_{i,u}t^s+\dfrac{\hbar}{2}\sum_{\substack{i>v\\s\geq0}}e_{i,v}t^{-s}e_{v,i}t^s\nonumber\\
&\quad+\dfrac{\hbar}{2}\sum_{\substack{u<i\\s\geq1}}e_{u,i}t^{-s}e_{i,u}t^s+\dfrac{\hbar}{2}\sum_{\substack{i<v\\s\geq1}}e_{i,v}t^{-s}e_{v,i}t^s+\hbar\sum_{s\geq0}\limits e_{i,i}t^{-s}e_{i,i}t^s.\label{align10}
\end{align}
By using this decomposition of $\ev_0(J(h_i))$, we have
\begin{align*}
\ev_0(B(h_i-h_{-i-2}))=K_i-K_{i+2},
\end{align*}
where
\begin{align}
K_i&=A_i+A_{-i}+\dfrac{\hbar}{8}[\sum_{\substack{u<v\\s\geq1}}(e_{u,v}-{(-1)}^{\widehat{u}+\widehat{v}}e_{-v,-u})t^{-s}(e_{v,u}-{(-1)}^{\widehat{u}+\widehat{v}}e_{-u,-v})t^s,e_{i,i}+e_{-i,-i}]\nonumber\\
&\quad+\dfrac{\hbar}{8}[\sum_{\substack{u>v\\s\geq0}}(e_{u,v}-{(-1)}^{\widehat{u}+\widehat{v}}e_{-v,-u})t^{-s}(e_{v,u}-{(-1)}^{\widehat{u}+\widehat{v}}e_{-u,-v})t^s,e_{i,i}+e_{-i,-i}].\label{align2.5.18}
\end{align}

Moreover, by we can rewrite this decomposition of $\ev_0(B(h_i-h_{-i-2}))$
\begin{align*}
&\quad((\bigotimes_{r=1}^l\limits\ev_{\xi_r})\circ\widetilde{\Delta}^l)(B(h_i-h_{-i-2}))\\
&=\square^l(K_i)-\square^l(K_{i+2})+C_i^l+C_{-i}^l-C_{i+2}^l-C_{-i-2}^l+F_i^l+F_{-i}^l-F_{i+2}^l-F_{-i-2}^l+G_i-G_{i+2}.
\end{align*}
Thus, it is enough to show that $\square^l(K_i)+C_i^l+C_{-i}^l+F_i^l+F_{-i}^l+G_i+\hbar(\widetilde{\mu}_D(W^{(2)}_{i,i}t+W^{(2)}_{-i,-i}t))$ can be written as the sum of the terms generated by $W^{(1)}_{i,j}t^s$.
\begin{Claim}\label{main claim}
We find that
\begin{equation*}
\square^l(K_i)+C_i^l+C_{-i}^l+F_i^l+F_{-i}^l+G_i+\hbar(\widetilde{\mu}_D(W^{(2)}_{i,i}t+W^{(2)}_{-i,-i}t))
\end{equation*}
is equal to
\begin{align*}
X&=\dfrac{\hbar}{2}\sum_{\substack{v\neq i\\s\geq0}}(e_{i,v}-{(-1)}^{\hat{i}+\hat{v}}e_{-v,-i})t^{-s}(e_{v,i}-{(-1)}^{\hat{i}+\hat{v}}e_{-i,-v})t^s\\
&\quad+\dfrac{\hbar}{2}\sum_{\substack{v\neq i\\s\geq1}}(e_{i,v}-{(-1)}^{\hat{i}+\hat{v}}e_{-v,-i})t^{-s}(e_{v,i}-{(-1)}^{\hat{i}+\hat{v}}e_{-i,-v})t^s\\
&\quad-\dfrac{\hbar}{2}(e_{i,i}-e_{-i,-i})^2+\hbar\sum_{s\geq0}\limits (e_{i,i}-e_{-i,-i})t^{-s}(e_{i,i}-e_{-i,-i})t^s.
\end{align*}
\end{Claim}
\begin{proof}
First, we compute the second-4th terms of \eqref{align2.5.18}.
By changing $i$ to $-i$ in \eqref{align10}, we have
\begin{align}
A_{-i}&=-\dfrac{\hbar}{2}e_{-i,-i}^2+\dfrac{\hbar}{2}\sum_{\substack{u<i\\s\geq0}}e_{-u,-i}t^{-s}e_{-i,-u}t^s+\dfrac{\hbar}{2}\sum_{\substack{i<v\\s\geq0}}e_{-i,-v}t^{-s}e_{-v,-i}t^s\nonumber\\
&\quad+\dfrac{\hbar}{2}\sum_{\substack{u>i\\s\geq1}}e_{-u,-i}t^{-s}e_{-i,-u}t^s+\dfrac{\hbar}{2}\sum_{\substack{i>v\\s\geq1}}e_{-i,-v}t^{-s}e_{-v,-i}t^s+\hbar\sum_{s\geq0}\limits e_{-i,-i}t^{-s}e_{-i,-i}t^s.\label{align11}
\end{align}
By a direct computation, we also obtain
\begin{align}
&\quad\dfrac{\hbar}{8}[\sum_{\substack{u<v\\s\geq1}}(e_{u,v}-{(-1)}^{\widehat{u}+\widehat{v}}e_{-v,-u})t^{-s}(e_{v,u}-{(-1)}^{\widehat{u}+\widehat{v}}e_{-u,-v})t^s,e_{i,i}+e_{-i,-i}]\nonumber\\
&=\dfrac{\hbar}{2}\sum_{\substack{u<i\\s\geq1}}{(-1)}^{\widehat{u}+\widehat{i}}e_{-i,-u}t^{-s}e_{i,u}t^s-\dfrac{\hbar}{2}\sum_{\substack{i<v\\s\geq1}}{(-1)}^{\widehat{i}+\widehat{v}}e_{-v,-i}t^{-s}e_{v,i}t^s\nonumber\\
&\quad+\dfrac{\hbar}{2}\sum_{\substack{u<-i\\s\geq1}}{(-1)}^{\widehat{u}+\widehat{i}+1}e_{i,-u}t^{-s}e_{-i,u}t^s-\dfrac{\hbar}{2}\sum_{\substack{-i<v\\s\geq1}}{(-1)}^{\widehat{i}+\widehat{v}+1}e_{-v,i}t^{-s}e_{v,-i}t^s,\label{align7}\\
&\quad\dfrac{\hbar}{8}[\sum_{\substack{u>v\\s\geq0}}(e_{u,v}-{(-1)}^{\widehat{u}+\widehat{v}}e_{-v,-u})t^{-s}(e_{v,u}-{(-1)}^{\widehat{u}+\widehat{v}}e_{-u,-v})t^s,e_{i,i}+e_{-i,-i}]\nonumber\\
&=\dfrac{\hbar}{2}\sum_{\substack{u>i\\s\geq0}}{(-1)}^{\widehat{u}+\widehat{i}}e_{-i,-u}t^{-s}e_{i,u}t^s-\dfrac{\hbar}{2}\sum_{\substack{i>v\\s\geq0}}{(-1)}^{\widehat{i}+\widehat{v}}e_{-v,-i}t^{-s}e_{v,i}t^s\nonumber\\
&\quad+\dfrac{\hbar}{2}\sum_{\substack{u>-i\\s\geq0}}{(-1)}^{\widehat{u}+\widehat{i}+1}e_{i,-u}t^{-s}e_{-i,u}t^s-\dfrac{\hbar}{2}\sum_{\substack{-i>v\\s\geq0}}{(-1)}^{\widehat{i}+\widehat{v}+1}e_{-v,i}t^{-s}e_{v,-i}t^s,\label{align8}
\end{align}

Next, we compute $C_i^l+C_{-i}^l+F_i^l+F_{-i}^l+G_i+\hbar(\widetilde{\mu}_D(W^{(2)}_{i,i}t+W^{(2)}_{-i,-i}t))$.
We denote the $i$-th term of $(\text{equation number})$ by $(\text{equation number})_i$. By a direct computation, we obtain
\begin{align}
&\quad\eqref{3.5.6}_1+\hbar\eqref{align5}_1=\hbar\sum_{\substack{1\leq r_1<r_2\leq l'\\s\in\mathbb{Z}}}e^{(r_1)}_{i,i}t^{-s}e^{(r_2)}_{i,i}t^s+\dfrac{\hbar}{2}\sum_{\substack{1\leq r_1<r_2\leq l'\\s\in\mathbb{Z},u\neq i}}e^{(r_1)}_{u,i}t^{-s}e^{(r_2)}_{i,u}t^s,\label{stolo}\\
&\quad\eqref{3.5.7}_1+\hbar\eqref{align5}_2=\hbar\sum_{\substack{1\leq r_1<r_2\leq l\\s\in\mathbb{Z}}}e^{(r_1)}_{-i,-i}t^{-s}e^{(r_2)}_{-i,-i}t^s+\dfrac{\hbar}{2}\sum_{\substack{1\leq r_1<r_2\leq l'\\s\in\mathbb{Z},u\neq-i}}e^{(r_1)}_{u,-i}t^{-s}e^{(r_2)}_{-i,u}t^s,\label{stolo1}\\
&\quad\eqref{3.5.8}_2+\eqref{3.5.9}_2+\hbar\eqref{align5}_3+\hbar\eqref{align5}_4+\hbar\eqref{align5}_5+\hbar\eqref{align5}_6\nonumber\\
&=-\hbar\sum_{\substack{1\leq r\leq l',\\u<i,s\geq0}}\limits{(-1)}^{\widehat{u}+\widehat{i}}e^{(r)}_{-i,-u}t^{-s}e^{(r)}_{i,u}t^s-\hbar\sum_{\substack{1\leq r\leq l',\\u<i,s\geq1}}\limits{(-1)}^{\widehat{u}+\widehat{i}}e^{(r)}_{i,u}t^{-s}e^{(r)}_{-i,-u}t^s\nonumber\\
&\quad-\hbar\sum_{\substack{1\leq r\leq l'\\u>i,s\geq0}}\limits{(-1)}^{\widehat{u}+\widehat{i}}e^{(r)}_{-i,-u}t^{-s}e^{(r)}_{i,u}t^s-\hbar\sum_{\substack{1\leq r\leq l'\\u>i,s\geq1}}\limits{(-1)}^{\widehat{u}+\widehat{i}}e^{(r_1)}_{i,u}t^{-s}e^{(r_1)}_{-i,-u}t^s\nonumber\\
&\quad-\dfrac{\hbar}{2}\sum_{\substack{1\leq r_1<r_2\leq l'\\u\neq i,s\in\mathbb{Z}}}{(-1)}^{\widehat{u}+\widehat{i}}e^{(r_1)}_{-i,-u}t^{-s} e^{(r_2)}_{i,u}t^s-\dfrac{\hbar}{2}\sum_{\substack{1\leq r_1<r_2\leq l'\\u\neq i,s\in\mathbb{Z}}}{(-1)}^{\widehat{u}+\widehat{i}}e^{(r_1)}_{i,u}t^{-s} e^{(r_2)}_{-i,-u}t^s,\label{stolo3}\\
&\quad\hbar\eqref{align5}_9+\hbar\eqref{align5}_{10}+G_i=0,\\
&\quad\hbar\eqref{align5}_7+\hbar\eqref{align5}_8\nonumber\\
&=-\sum_{\substack{1\leq r\leq l'\\s\geq0}}\limits e^{(r)}_{-i,-i}t^{-s}e^{(r)}_{i,i}t^s-\sum_{\substack{1\leq r\leq l'\\s\geq1}}\limits e^{(r)}_{i,i}t^{-s}e^{(r)}_{-i,-i}t^s\nonumber\\
&\quad-\sum_{\substack{1\leq r_1,r_2\leq l'\\s\geq0\\r_1\neq r_2}}\limits e^{(r_1)}_{-i,-i}t^{-s}e^{(r_2)}_{i,i}t^s-\sum_{\substack{1\leq r_1,r_2\leq l'\\s\geq1\\r_1\neq r_2}}\limits e^{(r_1)}_{i,i}t^{-s}e^{(r_2)}_{-i,-i}t^s\label{stolo4}
\end{align}
By a direct computation, we have
\begin{align}
&\quad\eqref{3.5.6}_2+\eqref{3.5.7}_2+\eqref{3.5.8}_1+\eqref{3.5.9}_1+\eqref{stolo}_2+\eqref{stolo1}_2+\eqref{stolo3}_5+\eqref{stolo3}_6\nonumber\\
&=(\widetilde{\Delta}^l-\square^l)\big(\dfrac{\hbar}{2}\sum_{\substack{v\neq i\\s\geq0}}(e_{i,v}-{(-1)}^{\hat{i}+\hat{v}}e_{-v,-i})t^{-s}(e_{v,i}-{(-1)}^{\hat{i}+\hat{v}}e_{-i,-v})t^s\big)\nonumber\\
&\quad+(\widetilde{\Delta}^l-\square^l)\big(\dfrac{\hbar}{2}\sum_{\substack{v\neq i\\s\geq1}}(e_{i,v}-{(-1)}^{\hat{i}+\hat{v}}e_{-v,-i})t^{-s}(e_{v,i}-{(-1)}^{\hat{i}+\hat{v}}e_{-i,-v})t^s\big),\label{ahoo1}\\
&\quad\eqref{stolo4}_3+\eqref{stolo3}_4+\eqref{stolo}_1+\eqref{stolo1}_1\nonumber\\
&=(\widetilde{\Delta}^l-\square^l)\big(-\dfrac{\hbar}{2}(e_{i,i}-e_{-i,-i})^2+\hbar\sum_{s\geq0}\limits (e_{i,i}-e_{-i,-i})t^{-s}(e_{i,i}-e_{-i,-i})t^s\big),\label{ahoo2}\\
&\quad\eqref{stolo3}_1+\eqref{stolo3}_2+\eqref{stolo3}_3+\eqref{stolo3}_4+\eqref{stolo4}_1+\eqref{stolo4}_2\nonumber\\
&=-\square^l(\hbar\sum_{\substack{u<i,s\geq0}}\limits{(-1)}^{\widehat{u}+\widehat{i}}e_{-i,-u}t^{-s}e_{i,u}t^s)\nonumber\\
&\quad-\square^l(\hbar\sum_{\substack{u<i,s\geq1}}\limits{(-1)}^{\widehat{u}+\widehat{i}}e_{i,u}t^{-s}e_{-i,-u}t^s)\nonumber\\
&\quad-\square^l(\hbar\sum_{\substack{u>i,s\geq0}}\limits{(-1)}^{\widehat{u}+\widehat{i}}e_{-i,-u}t^{-s}e_{i,u}t^s)-\square^l(\hbar\sum_{\substack{u>i,s\geq1}}\limits{(-1)}^{\widehat{u}+\widehat{i}}e_{i,u}t^{-s}e_{-i,-u}t^s)\nonumber\\
&\quad-\square^l(\hbar\sum_{\substack{s\geq0}}\limits e_{-i,-i}t^{-s}e_{i,i}t^s+\hbar\sum_{\substack{s\geq1}}\limits e_{i,i}t^{-s}e_{-i,-i}t^s).\label{ahoo3}
\end{align}
We note that $((\bigotimes_{r=1}^l\limits\ev_{\xi_r})\circ\widetilde{\Delta}^l)(B(h_i-h_{-i-2}))+\hbar(\widetilde{\mu}_D(W^{(2)}_{i,i}t+W^{(2)}_{-i,-i}t))$ is equal to the sum of $\square^l(\ev_0(B(h_i-h_{-i-2})))$ and \eqref{ahoo1}-\eqref{ahoo3} and the sum of \eqref{ahoo1} and \eqref{ahoo2} is equal to $X$.

Thus, in order to prove Claim~\ref{main claim}, it is enough to show that
\begin{align}
&\ev_0(B(h_i-h_{-i-2}))-\hbar\sum_{\substack{u<i,s\geq0}}\limits{(-1)}^{\widehat{u}+\widehat{i}}e_{-i,-u}t^{-s}e_{i,u}t^s-\hbar\sum_{\substack{u<i,s\geq1}}\limits{(-1)}^{\widehat{u}+\widehat{i}}e_{i,u}t^{-s}e_{-i,-u}t^s\nonumber\\
&\quad-\hbar\sum_{\substack{u>i,s\geq0}}\limits{(-1)}^{\widehat{u}+\widehat{i}}e_{-i,-u}t^{-s}e_{i,u}t^s-\hbar\sum_{\substack{u>i,s\geq1}}\limits{(-1)}^{\widehat{u}+\widehat{i}}e_{i,u}t^{-s}e_{-i,-u}t^s\nonumber\\
&\quad-\hbar\sum_{\substack{s\geq0}}\limits e_{-i,-i}t^{-s}e_{i,i}t^s+\hbar\sum_{\substack{s\geq1}}\limits e_{i,i}t^{-s}e_{-i,-i}t^s\label{align6}
\end{align}
is equal to $X$.
By a direct computation, we have
\begin{align*}
&\quad\eqref{align6}_4+\eqref{align8}_1+\eqref{align8}_4+\eqref{align10}_2+\eqref{align11}_3\\
&=\dfrac{\hbar}{2}\sum_{\substack{u>i\\s\geq0}}(e_{u,i}-{(-1)}^{\widehat{i}+\widehat{u}}e_{-i,-u})t^{-s}(e_{i,u}-{(-1)}^{\widehat{i}+\widehat{u}}e_{-u,-i})^s,\\
&\quad\eqref{align6}_2+\eqref{align8}_2+\eqref{align8}_3+\eqref{align10}_3+\eqref{align11}_4\\
&=\dfrac{\hbar}{2}\sum_{\substack{i>v\\s\geq0}}(e_{i,v}-{(-1)}^{\widehat{i}+\widehat{u}}e_{-v,-i})t^{-s}(e_{v,i}-{(-1)}^{\widehat{i}+\widehat{u}}e_{-i,-v})t^s,\\
&\quad\eqref{align6}_3+\eqref{align7}_1+\eqref{align7}_4+\eqref{align10}_4+\eqref{align11}_5\\
&=\dfrac{\hbar}{2}\sum_{\substack{u<i\\s\geq1}}(e_{u,i}-{(-1)}^{\widehat{i}+\widehat{u}}e_{-i,-u})t^{-s}(e_{i,u}-{(-1)}^{\widehat{i}+\widehat{u}}e_{-u,-i})t^s,\\
&\quad\eqref{align6}_5+\eqref{align7}_2+\eqref{align7}_3+\eqref{align10}_5+\eqref{align11}_2\\
&=\dfrac{\hbar}{2}\sum_{\substack{i<v\\s\geq1}}(e_{i,v}-{(-1)}^{\widehat{i}+\widehat{u}}e_{-v,-i})t^{-s}(e_{v,i}-{(-1)}^{\widehat{i}+\widehat{u}}e_{-i,-v})t^s,\\
&\quad\eqref{stolo}_1+\eqref{stolo1}_1+\eqref{align5}_7+\eqref{align5}_8+\eqref{align10}_1+\eqref{align10}_6+\eqref{align11}_1+\eqref{align11}_6\\
&=-\dfrac{\hbar}{2}(e_{i,i}-e_{-i,-i})^2+\hbar\sum_{s\geq0}\limits (e_{i,i}-e_{-i,-i})t^{-s}(e_{i,i}-e_{-i,-i})t^s.
\end{align*}
Adding the above five relations, we complete the proof of Claim~\ref{main claim}.
\end{proof}
By a direct computation, we obtain
\begin{align*}
X&=\dfrac{\hbar}{2}\sum_{\substack{v\neq i\\s\geq0}}\tilde{\mu}_D(W^{(1)}_{v,i}t^{-s})\tilde{\mu}_D(W^{(1)}_{i,v}t^s)\\
&\quad+\dfrac{\hbar}{2}\sum_{\substack{v\neq i\\s\geq1}}\tilde{\mu}_D(W^{(1)}_{v,i}t^{-s})\tilde{\mu}_D(W^{(1)}_{i,v}t^s)\\
&\quad-\dfrac{\hbar}{2}(\tilde{\mu}_D(W^{(1)}_{i,i}))^2+\hbar\sum_{s\geq0}\limits \tilde{\mu}_D(W^{(1)}_{i,i}t^{-s})\tilde{\mu}_D(W^{(1)}_{i,i}t^s).
\end{align*}
Thus, $X=\square^l(K_i)+C_i^l+C_{-i}^l+F_i^l+F_{-i}^l+G_i+\hbar(\widetilde{\mu}_D(W^{(2)}_{i,i}t+W^{(2)}_{-i,-i}t))$ can be written as the sum of the terms generated by $W^{(1)}_{i,j}t^s$.
This completes the proof of Theorem~\ref{main thm}.
\end{proof}
\begin{Theorem}
Provided that $\alpha\neq0$, the homomorphism $\Phi$ is surjective.
\end{Theorem}
\begin{proof}
We denote the image of $TY_{\ve_1,\ve_2}(\widehat{sp}(n))$ via $\Phi$ by $\text{Im}\Phi$. By Theorem~\ref{gener}, it is enough to show that $\{W^{(r)}_{i,j}t^s\mid 1\leq i,j\leq n,r=1,2,s\in\mathbb{Z}\}$ is contained in $\text{Im}\Phi$. By the definition of $\Phi(U(\widehat{\mathfrak{k}}))$, $\text{Im}\Phi$ contains $W^{(1)}_{j,i}t^s$ for all $i\neq j$. Take $(i,j)$ such that $i\neq \pm j,-i-2$. By the definition of $\Phi(B(h_i-h_{i+2}))$, we find that
\begin{align*}
\gamma_i&=(W^{(2)}_{i,i}-W^{(2)}_{i+2,i+2})t-\sum_{\substack{m\geq1}}\limits W^{(1)}_{i,i}t^{-m}W^{(1)}_{i,i}t^m-\dfrac{1}{2}(W^{(1)}_{i,i})^2\\
&\quad+\sum_{\substack{m\geq1}}\limits W^{(1)}_{i+2,i+2}t^{-m}W^{(1)}_{i+2,i+2}t^m+\dfrac{1}{2}(W^{(1)}_{i+2,i+2})^2
\end{align*}
is contained in $\text{Im}\Phi$. By Lemma~\ref{Lemma}, we find that $[\gamma_i,W^{(1)}_{j,i}t^s]$ is equal to
\begin{align}
\gamma_{i,s}&=(1+\delta_{j,i+2}-{(-1)}^{p(i)+p(j)}\delta_{-j,i+2})W^{(2)}_{j,i}t^{s+1}+\dfrac{l-1}{2}s\alpha(1+\delta_{j,i+2}+\delta_{-i,j}+\delta_{-j,i+2})W^{(1)}_{j,i}t^{s}\nonumber\\
&\quad+\dfrac{1}{2}s{W}^{(1)}_{-i,-j}t^s+\dfrac{1}{2}\delta_{j,-i-2}{W}^{(1)}_{-i-2,i}-\dfrac{1}{2}\delta_{j,i+2}{W}^{(1)}_{i+2,i}\nonumber\\
&\quad-\sum_{\substack{m\geq1}}\limits W^{(1)}_{j,i}t^{-m+s}W^{(1)}_{i,i}t^m-\sum_{\substack{m\geq1}}\limits W^{(1)}_{i,i}t^{-m}W^{(1)}_{j,i}t^{m+s}\nonumber\\
&\quad-\dfrac{1}{2}(1+\delta_{-i,j})W^{(1)}_{j,i}t^{s}W^{(1)}_{i,i}-\dfrac{1}{2}(1+\delta_{-i,j})W^{(1)}_{i,i}W^{(1)}_{j,i}t^{s}\nonumber\\
&\quad-(\delta_{j,i+2}-\delta_{j,-i-2})\sum_{\substack{m\geq1}}\limits W^{(1)}_{i+2,j}t^{-m+s}W^{(1)}_{i+2,i+2}t^m\nonumber\\
&\quad-(\delta_{j,i+2}-\delta_{j,-i-2})\sum_{\substack{m\geq1}}\limits W^{(1)}_{i+2,i+2}t^{-m}W^{(1)}_{i+2,j}t^{m+s}\nonumber\\
&\quad-\dfrac{1}{2}(\delta_{j,i+2}-\delta_{j,-i-2})(W^{(1)}_{j,i+2}t^{s}W^{(1)}_{i+2,i+2}+W^{(1)}_{i+2,i+2}W^{(1)}_{j,i+2}t^{s})\label{align43}
\end{align}
for all $i\neq j$. Then, by Lemma~\ref{Lemma}, we obtain
\begin{align*}
&\quad[W^{(1)}_{i,j}t,\gamma_{i,s}]-[W^{(1)}_{i,j},\gamma_{i,s+1}]\\
&=\dfrac{l-1}{2}\alpha(1+\delta_{j,i+2}-{(-1)}^{p(i)+p(j)}\delta_{-j,i+2})(W^{(1)}_{i,i}+W^{(1)}_{j,j})t^{s+1}-\dfrac{l}{2}\alpha W^{(1)}_{i,i}t^{s+1}\\
&\qquad\qquad\qquad\qquad\qquad\qquad\qquad\qquad\qquad\qquad\qquad\qquad+\text{completion of sum of terms of $U(\widehat{\mathfrak{k}})$}.
\end{align*}
Then, we find that $W^{(1)}_{i,i}t^{s+1}$ is contained in $\text{Im}\Phi$. Since $(W^{(1)}_{p,p}-W^{(1)}_{q,q})t^{s+1}$ is contained in $\text{Im}\Phi$ for all $p,q$, we find that $W^{(1)}_{p,p}t^{s+1}$ is contained in $\text{Im}\Phi$ for all $p$.

Since $W^{(1)}_{i,j}t^s$ is contained in $\text{Im}\Phi$ for all $i,j$, $(W^{(2)}_{i,i}-W^{(2)}_{j,j})t$ is contained in $\text{Im}\Phi$. By Lemma~\ref{Lemma} (1), we obtain
\begin{align}
&\quad[W^{(1)}_{j,i}t^s,(W^{(2)}_{i,i}-W^{(2)}_{j,j})t]\nonumber\\
&=(2+\delta_{i,-j}{(-1)}^{\hat{i}+\hat{j}})W^{(2)}_{j,i}t^{s+1}+\dfrac{l-1}{2}s\alpha(2+\delta_{i,-j})W^{(1)}_{j,i}t^{s}.\label{arutyu}
\end{align}
By \eqref{arutyu}, we find that $W^{(2)}_{i,j}t^s\ (i\neq j)$ is contained in $\text{Im}\Phi$ by \eqref{align43}. By using Lemma~\ref{Lemma} (1), we have
\begin{align*}
[W^{(1)}_{i,i+2},W^{(2)}_{i+2,i}t^s]
&=(1+\delta_{2i+2,0})(W^{(2)}_{i+2,i+2}-W^{(2)}_{i,i})t^s.
\end{align*}
Since $W^{(1)}_{i,i+2}$ and $W^{(2)}_{i+2,i}t^s$ are contained in $\text{Im}\Phi$, $(W^{(2)}_{i+2,i+2}-W^{(2)}_{i,i})t^s$ is contained in $\text{Im}\Phi$.
By using Lemma~\ref{Lemma} (1), we obtain
\begin{align*}
&\quad[(W^{(2)}_{i,i}-W^{(2)}_{i+2,i+2})t,(W^{(2)}_{i,i}-W^{(2)}_{i+2,i+2})t^{s}]-[(W^{(2)}_{i,i}-W^{(2)}_{i+2,i+2}),(W^{(2)}_{i,i}-W^{(2)}_{i+2,i+2})t^{s+1}]\\
&=\alpha(W^{(2)}_{i,i}+W^{(2)}_{i+2,i+2})t^s
\end{align*}
for all $\widehat{i}=\widehat{i+2}$. By the assumption $\alpha\neq0$, $(W^{(2)}_{i,i}+W^{(2)}_{i+2,i+2})t^s$ is contained in $\text{Im}\Phi$. Since we have already shown that $(W^{(2)}_{i,i}-W^{(2)}_{i+2,i+2})t^s$ is contained in $\text{Im}\Phi$, $W^{(2)}_{i,i}t^s$ is contained in $\text{Im}\Phi$. This completes the proof.
\end{proof}

\appendix
\section{Generators of rectangular $W$-algebras of type $D$}
This section is devoted to the proof of Theorem~\ref{gener}.
We define a grading on $\mathfrak{b}$ by setting $\text{deg}(x)=j$ if $x\in\mathfrak{b}\cap\mathfrak{g}_{j}$. For $a,b\in I_{nl}$, let $\gamma_{a,b}$ be $\sum_{0< 2u\leq q-p}\limits\widehat{q+2u}+\widehat{p}\cdot\widehat{j}+\widehat{q}\cdot\widehat{i}$, where $p=\col(a),q=\col(b),j=\row(a),i=\row(b)$. Since
\begin{equation*}
\{\sum_{\substack{\row(a)=j,\row(b)=i,\\\col(a)=\col(b)+2s}}{(-1)}^{\gamma_{a,b}}f_{a,b}\mid0\leq s\leq l-1,1\leq i,j\leq n\}
\end{equation*}
forms a basis of $\mathfrak{so}(nl)^f=\{g\in\mathfrak{so}(nl)|[f,g]=0\}$, it is enough to show that $W^{(1)}_{i,j}$ and $W^{(2)}_{i,j}$ generate the term whose form is
\begin{equation*}
\sum_{\substack{\row(a)=j,\row(b)=i,\\\col(a)=\col(b)+2s}}{(-1)}^{\gamma_{a,b}}f_{a,b}[-1]+\text{higher terms}
\end{equation*}
for all $0\leq s\leq l-1,\ 1\leq i,j\leq n$ by Theorem~4.1 of \cite{KW1}. The proof is completed by two claims, that is, Lemma~\ref{Cl1} and Lemma~\ref{Cl2}. In order to simplify computations, we prepare the following notations. Let us set
\begin{gather*}
Z_{i,i}=\sum_{\substack{\row(a)=i,\row(b)=i,\\\col(a)=\col(b)+2=p}}{(-1)}^{\widehat{p}+\widehat{p}\cdot\widehat{i}+\widehat{p-2}\cdot\widehat{i}}f_{a,b}[-1],\quad
V_{i,i}=W^{(2)}_{i,i}-Z_{i,i}.
\end{gather*}
Then, $Z_{i,i}$ is a degree $-2$ term and $V_{i,i}$ is a degree $-1$ term. We also denote the condition that $\row(a)=i,\row(b)=j,\col(a)=\col(b)+2$ by $(A)_{i,j}$, the condition that
\begin{gather*}
\row(a_2)=i,\row(b_1)=j,p=\col(a_1)=\col(b_1)<\col(a_2)=\col(b_2)=q,\row(a_1)=\row(b_2)=r
\end{gather*}
by $(B)_{i,j}$, the condition that $\row(a)=i,\row(b)=j,\col(a)=\col(b)+2s$
by $(C)_{i,j}^s$, and the condition that $\row(c)=i,\row(d)=j,\col(c)=\col(d)+2s$
by $(D)_{i,j}^s$. Moreover, for all $a_i\in V^\kappa(\mathfrak{gl}(n))^{\otimes l}$ and $s_i\in\mathbb{Z}$, we set
\begin{equation*}
(a_1)_{(s_1)}(a_2)_{(s_2)}\cdots (a_{u-1})_{(s_{u-1})}a_{u}=(a_1)_{(s_1)}\Big((a_2)_{(s_2)}\big(\cdots ((a_{u-1})_{(s_{u-1})}a_{u})\cdots\big)\Big).
\end{equation*}
\begin{Lemma}\label{Cl1}
\textup{(1)}\ For all $i\neq j$, $\{W^{(r)}_{p,q}\mid 1\leq p.q\leq n,r=1,2\}$ generate
\begin{equation*}
\sum_{\substack{(C)_{j,i}^s}}{(-1)}^{\gamma_{a,b}}f_{a,b}[-1]+\text{higher terms}.
\end{equation*}
\textup{(2)}\ For all $i\neq j$, $\{W^{(r)}_{p,q}\mid 1\leq p,q\leq n,r=1,2\}$ generate
\begin{equation*}
\sum_{\substack{(C)_{i,i}^s}}{(-1)}^{\gamma_{a,b}}f_{a,b}[-1]-\sum_{\substack{(C)_{j,j}^s}}{(-1)}^{\gamma_{a,b}}f_{a,b}[-1]+\text{higher terms}.
\end{equation*}
\end{Lemma}
\begin{proof}
\textup{(1)}\ By a direct computation, we obtain
\begin{align}
&\quad(\sum_{\substack{(C)_{j,i}^s}}{(-1)}^{\gamma_{a,b}}f_{a,b}[-1])_{(0)}\sum_{\substack{(C)_{v,u}^s}}{(-1)}^{\gamma_{a,b}}f_{a,b}[-1]\nonumber\\
&=\delta_{i,v}\sum_{\substack{(C)_{j,u}^{s+t}}}{(-1)}^{\gamma_{a,b}}f_{a,b}[-1]-\delta_{j,u}\sum_{\substack{(C)_{v,i}^{s+t}}}{(-1)}^{\gamma_{a,b}}f_{a,b}[-1]\nonumber\\
&\quad-\delta_{-j,v}\sum_{\substack{(C)_{-i,u}^{s+t}}}{(-1)}^{s+\hat{i}+\hat{j}+\gamma_{a,b}}f_{a,b}[-1]+\delta_{i,-u}\sum_{\substack{(C)_{v,-j}^{s+t}}}{(-1)}^{s+\hat{i}+\hat{j}+\gamma_{a,b}}f_{a,b}[-1].\label{y1}
\end{align}
By \eqref{y1}, we have the following equation;
\begin{align*}
((W^{(2)}_{j,j})_{(0)})^sW^{(1)}_{i,j}&=(Z_{j,j})^sW^{(1)}_{i,j}+\text{higher terms}\\
&=\sum_{\substack{\row(a)=j,\row(b)=i,\\\col(a)=\col(b)+2s}}{(-1)}^{\gamma_{a,b}}f_{a,b}[-1]+\text{higher terms}
\end{align*}
for all $i\neq-j, j$. Then, we have proven (1) in the case that $j\neq i,-i$. Taking $p\in I_{n}$ such that $i\neq \pm p$,
we obtain
\begin{align*}
(W^{(1)}_{p,i})_{(0)}\sum_{\substack{\row(a)=p,\row(b)=-i,\\\col(a)=\col(b)+s}}{(-1)}^{\gamma_{a,b}}f_{a,b}[-1]
&=\sum_{\substack{\row(a)=p,\row(b)=-i,\\\col(a)=\col(b)+s}}{(-1)}^{\gamma_{a,b}}f_{a,b}[-1]
\end{align*}
by \eqref{y1}. We have shown (1) in the case that $j=-i$. This completes the proof of (1).

\textup{(2)}\ It is enough to show the case when $i\neq \pm j$ since the case that $i=-j$ is naturally derived from other cases. By \eqref{y1}, we obtain
\begin{align*}
&\quad(W^{(1)}_{i,j})_{(0)}\sum_{\substack{\row(a)=i,\row(b)=j,\\\col(a)=\col(b)+s}}{(-1)}^{\gamma_{a,b}}f_{a,b}[-1]\\
&=\sum_{\substack{\row(a)=j,\row(b)=j,\\\col(a)=\col(b)+s}}{(-1)}^{\hat{i}+\hat{j}+\gamma_{a,b}}f_{a,b}[-1]-\sum_{\substack{\row(a)=i,\row(b)=i,\\\col(a)=\col(b)+s}}{(-1)}^{\hat{i}+\hat{j}+\gamma_{a,b}}f_{a,b}[-1]
\end{align*}
for all $i\neq \pm j$. TWe have shown (2).
\end{proof}

\begin{Lemma}\label{Cl2}
Suppose that $j\neq \pm i$. We obtain
\begin{align*}
&\quad(W^{(2)}_{i,i})_{(1)}(W^{(1)}_{i,j})_{(0)}((W^{(2)}_{i,i})_{(0)})^sW^{(1)}_{j,i}\nonumber\\
&=-s\sum_{\substack{\row(a)=i,\row(b)=i,\\\col(a)=\col(b)+2s}}{(-1)}^{\gamma_{a,b}}f_{a,b}[-1]+s\sum_{\substack{\row(a)=j,\row(b)=j,\\\col(a)=\col(b)+2s}}{(-1)}^{\gamma_{a,b}}f_{a,b}[-1]\nonumber\\
&\quad+\alpha\sum_{\substack{\row(a)=i,\row(b)=i,\\\col(a)=\col(b)+2s}}{(-1)}^{\gamma_{a,b}}f_{a,b}[-1].
\end{align*}
\end{Lemma}
\begin{proof}
By the degree of $Z_{i,i}$ and $V_{i,i}$, we obtain
\begin{align}
&\quad(W^{(2)}_{i,i})_{(1)}(W^{(1)}_{i,j})_{(0)}((W^{(2)}_{i,i})_{(0)})^sW^{(1)}_{j,i}\nonumber\\
&=(Z_{i,i})_{(1)}(W^{(1)}_{i,j})_{(0)}((Z_{i,i})_{(0)})^sW^{(1)}_{j,i}+(V_{i,i})_{(1)}(W^{(1)}_{i,j})_{(0)}((Z_{i,i})_{(0)})^sW^{(1)}_{j,i}\nonumber\\
&\quad+\sum_{1\leq t\leq s}(Z_{i,i})_{(1)}(W^{(1)}_{i,j})_{(0)}((Z_{i,i})_{(0)})^{s-t}(V_{i,i})_{(0)}((Z_{i,i})_{(0)})^{t-1}W^{(1)}_{j,i}+\text{higher terms}.\label{aku}
\end{align}
We compute each terms of the right hand side of \eqref{aku}. Let us compute the first term of the right hand side of \eqref{aku}.
By \eqref{y1}, we obtain
\begin{align}
(W^{(1)}_{i,j})_{(0)}((Z_{i,i})_{(0)})^sW^{(1)}_{j,i}
&=\sum_{\substack{(C)_{j,j}^s}}{(-1)}^{\gamma_{a,b}}f_{a,b}[-1]-\sum_{\substack{(C)_{i,i}^s}}{(-1)}^{\gamma_{a,b}}f_{a,b}[-1].\label{asi}
\end{align}
By \eqref{y1} and \eqref{asi}, we have
\begin{gather}
(Z_{i,i})_{(1)}(W^{(1)}_{i,j})_{(0)}((Z_{i,i})_{(0)})^sW^{(1)}_{j,i}=0.\label{aku1}
\end{gather}
Next, let us compute the second term of the right hand side of \eqref{aku}.
By \eqref{asi}, we obtain
\begin{align}
&\quad(V_{i,i})_{(1)}(W^{(1)}_{i,j})_{(0)}((Z_{i,i})_{(0)})^sW^{(1)}_{j,i}\nonumber\\
&=\sum_{\substack{(A)_{i,i},(C)_{j,j}^s}}({(-1)}^{(\widehat{r}+\widehat{i})\cdot(\widehat{p}+\widehat{q})+\gamma_{a,b}}f_{a_1,b_1}[-1]f_{a_2,b_2}[-1])_{(1)}f_{a,b}[-1]\nonumber\\
&\quad-\sum_{\substack{(A)_{i,i},(C)_{i,i}^s}}({(-1)}^{(\widehat{r}+\widehat{i})\cdot(\widehat{p}+\widehat{q})+\gamma_{a,b}}f_{a_1,b_1}[-1]f_{a_2,b_2}[-1])_{(1)}f_{a,b}[-1]\nonumber\\
&\quad+(\alpha\sum_{\substack{(C)_{i,i}^p}}\dfrac{p}{2}f_{a,b}[-2])_{(1)}\sum_{\substack{(C)_{i,i}^s}}{(-1)}^{\gamma_{a,b}}f_{a,b}[-1].\label{arf}
\end{align}
By a direct computation, we obtain
\begin{align*}
&\quad\text{the first term of \eqref{arf}}=\text{the second term of \eqref{arf}}\\
&=\sum_{\substack{\row(a)=i,\row(b)=i,\\\col(a)=\col(b)+s=p}}{(-1)}^{(\widehat{j}+\widehat{i})\cdot(\widehat{p}+\widehat{p-s})+\gamma_{a,b}}f_{a,b}[-1],
\end{align*}
\begin{align*}
&\text{the third term of \eqref{arf}}=\alpha\sum_{\substack{(C)_{i,i}^s}}{(-1)}^{\gamma_{a,b}}sf_{a,b}[-1].
\end{align*}
Thus, we obtain
\begin{align}
&(V_{i,i})_{(1)}(W^{(1)}_{i,j})_{(0)}((Z_{i,i})_{(0)})^sW^{(1)}_{j,i}=\alpha\sum_{\substack{(C)_{i,i}^s}}{(-1)}^{\gamma_{a,b}}sf_{a,b}[-1].\label{aku2}
\end{align}
Next, let us compute the third term of \eqref{aku}.
By \eqref{y1}, we obtain
\begin{align*}
&((Z_{i,i})_{(0)})^{t-1}W^{(1)}_{j,i}=\sum_{\substack{(C)_{i,i}^{t-1}}}{(-1)}^{\gamma_{a,b}}f_{a,b}[-1].
\end{align*}
Since
\begin{align*}
(V_{i,i})_{(0)}=(\sum_{(A)_{i,i}}{(-1)}^{(\widehat{r}+\widehat{i})\cdot\widehat{p}+(\widehat{i}+\widehat{r})\cdot\widehat{q}}f_{a_1,b_1}[-1]f_{a_2,b_2}[-1])_{(0)}
\end{align*}
holds, we can rewrite $(V_{i,i})_{(0)}((Z_{i,i})_{(0)})^{t-1}W^{(1)}_{j,i}$ as
\begin{align*}
&\sum_{\substack{(A)_{i,i},(C)_{i,j}^{t-1}}}{(-1)}^{\beta_1}f_{a_1,b_1}[-1][f_{a_2,b_2},f_{a,b}][-1]+\sum_{\substack{(A)_{i,i},(C)_{i,j}^{t-1}}}{(-1)}^{\beta_1}f_{a_2,b_2}[-1][f_{a_1,b_1},f_{a,b}][-1],
\end{align*}
where $\beta_1=\gamma_{a,b}+(\widehat{r}+\widehat{i})\cdot(\widehat{p}+\widehat{q})$
such that $\row(a_1)=r,\col(a_1)=p,\col(a_2)=q$.
By a direct computation, we can rewrite $(V_{i,i})_{(0)}((Z_{i,i})_{(0)})^{t-1}W^{(1)}_{j,i}$ as
\begin{align}
&\sum_{\substack{(A)_{i,i},(C)_{i,j}^{t-1}}}{(-1)}^{\beta_1}\delta_{b_2,a}f_{a_1,b_1}[-1]f_{a_2,b}[-1]+\sum_{\substack{(A)_{i,i},(C)_{i,j}^{t-1}}}{(-1)}^{\beta_1}\delta_{b_2,-b}f_{a_1,b_1}[-1]f_{a,-a_2}[-1]\nonumber\\
&\quad+\sum_{\substack{(A)_{i,i},(C)_{i,j}^{t-1}}}{(-1)}^{\beta_1}\delta_{b_1,a}f_{a_2,b_2}[-1]f_{a_1,b}[-1]-\sum_{\substack{(A)_{i,i},(C)_{i,j}^{t-1}}}{(-1)}^{\beta_1}\delta_{a_1,b}f_{a_2,b_2}[-1]f_{a,b_1}[-1]\nonumber\\
&\quad-\sum_{\substack{(A)_{i,i},(C)_{i,j}^{t-1}}}{(-1)}^{\beta_1}\delta_{a_1,-a}f_{a_2,b_2}[-1]f_{-b_1,b}[-1]+\delta_{t,1}\sum_{\substack{(C)_{i,j}^{t-1}}}\alpha{(-1)}^{\gamma_{a,b}}\dfrac{(\col(b)-1+n)}{2}f_{a,b}[-2].\label{ti2}
\end{align}
Let us denote the sum of the first five terms of \eqref{ti2} by $B_t$. We can rewrite
\begin{equation*}
(Z_{i,i})_{(1)}(W^{(1)}_{i,j})_{(0)}((Z_{i,i})_{(0)})^{s-t}(V_{i,i})_{(0)}((Z_{i,i})_{(0)})^{t-1}W^{(1)}_{j,i}
\end{equation*}
as
\begin{align}
&-\displaystyle\sum_{g=0}^{s-t}\limits \begin{pmatrix} r-t\\g\end{pmatrix}((Z_{i,i})_{(0)})^{s-t-g}(\displaystyle\sum_{(D)^g_{j,i}}\limits {(-1)}^{\gamma_{c,d}}f_{c,d}[-1])_{(1)}B_t+(W^{(1)}_{i,j})_{(0)}((Z_{i,i})_{(0)})^{s-t}((Z_{i,i})_{(1)})B_t\nonumber\\
&\quad+(Z_{i,i})_{(1)}(W^{(1)}_{i,j})_{(0)}((Z_{i,i})_{(0)})^{s}\delta_{t,1}\sum_{\substack{(C)_{i,j}^{t-1}}}\alpha{(-1)}^{\gamma_{a,b}}\dfrac{(\col(b)-1+n)}{2}f_{a,b}[-2].\label{ti}
\end{align}
Let us compute each terms of \eqref{ti}. By a direct computation, we obtain
\begin{align}
\text{the third term of \eqref{ti}}
&=-\delta_{t,1}\alpha\sum_{\substack{\row(a)=i,\row(b)=i,\\\col(a)=\col(b)+2s}}{(-1)}^{\gamma_{a,b}}(s-1)f_{a,b}[-1].\label{equ3}
\end{align}
Next, we compute the first term of \eqref{ti}.
By \eqref{ti2}, we can rewrite $(\displaystyle\sum_{(D)^g_{j,i}}\limits {(-1)}^{\gamma_{c,d}}f_{c,d}[-1])_{(1)}B_t$ as
\begin{align}
&\sum_{\substack{(A)_{i,i},(C)_{i,j}^{t-1},(D)_{i,i}^g}}{(-1)}^{\beta_1+\gamma_{c,d}}\delta_{b_2,a}[[f_{c,d}, f_{a_1,b_1}],f_{a_2,b}][-1]\nonumber\\
&\quad+\sum_{\substack{(A)_{i,i},(C)_{i,j}^{t-1},(D)_{i,i}^g}}{(-1)}^{\beta_1+\gamma_{c,d}}\delta_{b_2,-b}[[f_{c,d},f_{a_1,b_1}],f_{a,-a_2}][-1]\nonumber\\
&\quad+\sum_{\substack{(A)_{i,i},(C)_{i,j}^{t-1},(D)_{i,i}^g}}{(-1)}^{\beta_1+\gamma_{c,d}}\delta_{b_1,a}[[f_{c,d},f_{a_2,b_2}],f_{a_1,b}][-1]\nonumber\\
&\quad-\sum_{\substack{(A)_{i,i},(C)_{i,j}^{t-1},(D)_{i,i}^g}}{(-1)}^{\beta_1+\gamma_{c,d}}\delta_{a_1,b}[[f_{c,d},f_{a_2,b_2}],f_{a,b_1}][-1]\nonumber\\
&\quad-\sum_{\substack{(A)_{i,i},(C)_{i,j}^{t-1},(D)_{i,i}^g}}{(-1)}^{\beta_1+\gamma_{c,d}}\delta_{a_1,-a}[[f_{c,d},f_{a_2,b_2}]f_{-b_1,b}][-1].\label{ti3}
\end{align}
We compute each terms of the right hand side of \eqref{ti3}. 
By a direct computation, we obtain
\begin{align}
\text{the first term of \eqref{ti3}}&=-\sum_{\substack{(A)_{i,i},(C)_{i,j}^{t-1},(D)_{i,i}^g}}{(-1)}^{\beta_1+\gamma_{c,d}}\delta_{b_2,a}\delta_{d,a_1}\delta_{c,b}f_{a_2,b_1}[-1],\label{1-1}\\
\text{the second term of \eqref{ti3}}&=0,\label{1-2}\\
\text{the third term of \eqref{ti3}}
&=-\sum_{\substack{(A)_{i,i},(C)_{i,j}^{t-1},(D)_{i,i}^g}}{(-1)}^{\beta_1+\gamma_{c,d}}\delta_{b_1,a}\delta_{d,-b_2}\delta_{a_2,-a_1}f_{c,b}[-1]\nonumber\\
&\quad-\sum_{\substack{(A)_{i,i},(C)_{i,j}^{t-1},(D)_{i,i}^g}}{(-1)}^{\beta_1+\gamma_{c,d}}\delta_{b_1,a}\delta_{d,-b_2}\delta_{c,b}f_{a_1,-a_2}[-1],\label{1-3}\\
\text{the 4-th term of \eqref{ti3}}&=\sum_{\substack{(A)_{i,i},(C)_{i,j}^{t-1},(D)_{i,i}^g}}{(-1)}^{\beta_1+\gamma_{c,d}}\delta_{a_1,b}\delta_{b_2,c}\delta_{d,a}f_{a_2,b_1}[-1],\label{1-4}\\
\text{the 5-th term of \eqref{ti3}}
&=\sum_{\substack{(A)_{i,i},(C)_{i,j}^{t-1},(D)_{i,i}^g}}{(-1)}^{\beta_1+\gamma_{c,d}}\delta_{a_1,-a}\delta_{d,a_2}\delta_{b_2,-b_1}f_{c,b}[-1]\nonumber\\
&\quad+\sum_{\substack{(A)_{i,i},(C)_{i,j}^{t-1},(D)_{i,i}^g}}{(-1)}^{\beta_1+\gamma_{c,d}}\delta_{a_1,-a}\delta_{d,a_2}\delta_{b,c}f_{-b_1,b_2}[-1].\label{1-5}
\end{align}
Since
\begin{gather*}
\eqref{1-1}=-\eqref{1-4},\quad
\text{the first term of }\eqref{1-3}=-\text{the first term of }\eqref{1-5},\\
\text{the second term of }\eqref{1-3}=-\text{the second term of }\eqref{1-5}
\end{gather*}
hold, we obtain
\begin{equation}\label{equ1}
\text{the first term of \eqref{ti}}=0
\end{equation}
by adding \eqref{1-1}-\eqref{1-5}.

Next, let us compute the second term of \eqref{ti}. By a direct computation, we also obtain
\begin{align}
&\quad((Z_{i,i})_{(0)})^{s-t}((Z_{i,i})_{(1)})B_t\nonumber\\
&=\sum_{\substack{(A)_{i,i},(C)_{i,j}^{t-1},(D)_{i,i}^1}}{(-1)}^{\beta_1+\gamma}\delta_{b_2,a}[[f_{c,d},f_{a_1,b_1}],f_{a_2,b}][-1]\nonumber\\
&\quad+\sum_{\substack{(A)_{i,i},(C)_{i,j}^{t-1},(D)_{i,i}^1}}{(-1)}^{\beta_1+\gamma_{c,d}}\delta_{b_2,-b}[[f_{c,d},f_{a_1,b_1}],f_{a,-a_2}][-1]\nonumber\\
&\quad+\sum_{\substack{(A)_{i,i},(C)_{i,j}^{t-1},(D)_{i,i}^1}}{(-1)}^{\beta_1+\gamma_{c,d}}\delta_{b_1,a} [[f_{c,d},f_{a_2,b_2}],f_{a_1,b}][-1]\nonumber\\
&\quad-\sum_{\substack{(A)_{i,i},(C)_{i,j}^{t-1},(D)_{i,i}^1}}{(-1)}^{\beta_1+\gamma_{c,d}}\delta_{a_1,b}[[f_{c,d},f_{a_2,b_2}],f_{a,b_1}][-1]\nonumber\\
&\quad-\sum_{\substack{(A)_{i,i},(C)_{i,j}^{t-1},(D)_{i,i}^1}}{(-1)}^{\beta_1+\gamma_{c,d}} \delta_{a_1,-a}[[f_{c,d},f_{a_2,b_2}],f_{-b_1,b}][-1].\label{ti4}
\end{align}
Let us compute each terms of \eqref{ti4}. By a direct computation, we obtain
\begin{align}
\text{the first term of \eqref{ti4}}&=0,\label{2-1}\\
\text{the second term of \eqref{ti4}}&=\sum_{\substack{(A)_{i,i},(C)_{i,j}^{t-1},(D)_{i,i}^1}}{(-1)}^{\beta_1+\gamma_{c,d}}\delta_{b_2,-b}\delta_{b_1,c}\delta_{d,a}f_{a_1,-a_2}[-1],\label{2-2}\\
\text{the third term of \eqref{ti4}}&=\sum_{\substack{(A)_{i,i},(C)_{i,j}^{t-1},(D)_{i,i}^1}}{(-1)}^{\beta_1+\gamma_{c,d}}\delta_{b_1,a} \delta_{b_2,c}\delta_{d,a_1}f_{a_2,b}[-1]\nonumber\\
&\quad-\sum_{\substack{(A)_{i,i},(C)_{i,j}^{t-1},(D)_{i,i}^1}}{(-1)}^{\beta_1+\gamma_{c,d}} \delta_{d,-b_2}\delta_{-c,a_1}f_{a_2,b}[-1],\label{2-3}\\
\text{the 4-th term of \eqref{ti4}}&=0,\label{2-4}\\
\text{the 5-th term of \eqref{ti4}}
&=\sum_{\substack{(A)_{i,i},(C)_{i,j}^{t-1},(D)_{i,i}^1}}{(-1)}^{\beta_1+\gamma_{c,d}} \delta_{a_1,-a}\delta_{d,a_2}\delta_{b_2,-b_1}f_{c,b}[-1]\nonumber\\
&\quad-\sum_{\substack{(A)_{i,i},(C)_{i,j}^{t-1},(D)_{i,i}^2}}{(-1)}^{\beta_1+\gamma}\delta_{d,-b_2}\delta_{c,b_1}f_{a_2,b}[-1].\label{2-5}
\end{align}
Since
\begin{gather*}
\eqref{2-2}=-\text{the first term of \eqref{2-5}},\quad\eqref{2-3}_2=-\text{the second term of \eqref{2-5}},\\
\eqref{2-3}_1=\sum_{\substack{\row(a)=i,\row(b)=j,\\\col(a)=\col(b)+s}}{(-1)}^{\gamma_{a,b}}f_{a,b}[-1]
\end{gather*}
hold, we obtain
\begin{equation}\label{equ2}
\text{the second term of \eqref{ti}}=\sum_{\substack{\row(a)=i,\row(b)=j,\\\col(a)=\col(b)+s}}{(-1)}^{\gamma_{a,b}}f_{a,b}[-1].
\end{equation}
by adding \eqref{2-1}-\eqref{2-5}.
Adding \eqref{equ3}, \eqref{equ1} and \eqref{equ2}, we obtain
\begin{align}
&\quad\text{the third term of \eqref{aku}}\nonumber\\
&=-s\sum_{\substack{\row(a)=i,\row(b)=i,\\\col(a)=\col(b)+2s}}{(-1)}^{\gamma_{a,b}}f_{a,b}[-1]+s\sum_{\substack{\row(a)=j,\row(b)=j,\\\col(a)=\col(b)+2s}}{(-1)}^{\gamma_{a,b}}f_{a,b}[-1]\nonumber\\
&\quad-\alpha\sum_{\substack{\row(a)=i,\row(b)=i,\\\col(a)=\col(b)+2s}}{(-1)}^{\gamma_{a,b}}(s-1)f_{a,b}[-1]\label{aku3}
\end{align}
by \eqref{y1}. Adding \eqref{aku1}, \eqref{aku2} and \eqref{aku3}, we obtain the proof.
\end{proof}
\section{Data availability statement}
The datasets generated during the current study are available from the corresponding author on reasonable request.
\bibliographystyle{plain}
\bibliography{syuu}
\end{document}